\documentclass[12pt]{amsart}
\usepackage[english]{babel}
\usepackage{graphicx,amssymb,epsfig,amsmath}
\usepackage{hyperref}  
\newcounter{assump}
\usepackage{lipsum}
\usepackage{enumitem} 
\usepackage{hyperref}
\usepackage[all]{xy}
\usepackage{algorithm}
\usepackage{algpseudocode}
\usepackage{xcolor}

\usepackage{color}

\newtheorem{theorem}{Theorem}[section]
\newtheorem{theo}[theorem]{Theorem}

\newtheorem {coro}[theorem]{Corollary}
\newtheorem {pro}[theorem]{Proposition}
\newtheorem{defi}[theorem]{Definition}

\newtheorem{lem}[theorem]{Lemma}

\newdimen\AAdi%
\newbox\AAbo%
\def\AAk#1#2{\setbox\AAbo=\hbox{#2}\AAdi=\wd\AAbo\kern#1\AAdi{}}%
\def\AAr#1#2#3{\setbox\AAbo=\hbox{#2}\AAdi=\ht\AAbo\raise#1\AAdi\hbox{#3}}%
\RequirePackage[OT1]{fontenc}
\usepackage{hyperref}
\newdimen\AAdi%
\newbox\AAbo%
\def\AAk#1#2{\setbox\AAbo=\hbox{#2}\AAdi=\wd\AAbo\kern#1\AAdi{}}%
\def\AAr#1#2#3{\setbox\AAbo=\hbox{#2}\AAdi=\ht\AAbo\raise#1\AAdi\hbox{#3}}%
\def\bbx{{\mathbb X}}

\usepackage{graphicx} 

\title[Gumbel Convergence]{Gumbel Convergence for Maximal Packing Distances from Uniform Random Samples}

\author{Sana Louhichi} 
\address{Univ. Grenoble Alpes, CNRS, Grenoble INP, LJK 38000 Grenoble, France.
700 Avenue Centrale,
38401 Saint-Martin-d'H\`eres, France.} 
\email{sana.louhichi@univ-grenoble-alpes.fr} 
\date{}

\begin{document}

\begin{abstract}
Let $(X_n)_n$ be a sequence of i.i.d.\ $\mathbb{R}^d$-valued random variables uniformly distributed on a compact subset $\mathbb{M}$ of  \(\mathbb{R}^d\). 
In this work, we study the asymptotic behavior of maximal packings of this compact subset $\mathbb M$. Under some assumptions on $\mathbb{M}$ {\textcolor{black}{(in particular, that it is a smooth submanifold of $\mathbb{R}^d$ without boundary)}}, we establish a precise extreme value law (of Gumbel type) for the minimal distances from the sample $(X_1,\cdots,X_n)$ to the centers of maximal packings. This result leads to explicit asymptotic confidence bounds for this maximal packing and thus to the support $\mathbb M$.

A distinctive feature of our contribution is the \emph{explicit derivation} of the scaling sequences  in the Gumbel convergence, depending only on the geometry of the support $\mathbb M$. These formulas provide  interpretable confidence bounds for $\mathbb M$, which represent a novel complement to previous approaches such as those by Fasy et al.\ (2014) \cite{Fasy}.

Our approach bridges geometric and combinatorial probability arguments and relies on analytic tools such as the Lambert W function.
In addition, we provide {\textcolor{black}{examples of main sample spaces in directional and circular statistics (circle, sphere, torus)}}, along with simulations that illustrate and support the theoretical findings.

\end{abstract}

\maketitle

\medskip


\vspace{0.5cm} 

\noindent {\small {\bf{ 2020 Mathematics Subject Classification: }} {60D05, 60G70, 60F05, 62G05, 62G15, 52C17.} 

\section{Introduction}

Understanding the asymptotic behavior of distances between random points and deterministic sets is a fundamental problem in stochastic geometry and probability theory. Such questions naturally arise in various contexts including, for instance, statistical estimation with confidence sets \cite{Eddie,CR2004,CR2009,Devroye,Fasy,smale}, covering problems \cite{Penrose, Penrose2025b}, and rates of convergence for persistence diagrams \cite{Ch,Fasy}.

In this setting, the Hausdorff distance provides a natural way to quantify how well a finite sample approximates an underlying deterministic set. Another notion useful for studying this problem {\textcolor{black}{(and which will be the main focus of this paper)}}  is that of maximal packings of a compact set. We briefly recall these two notions.

Let $A$ and $B$ be two {closed} subsets of $\mathbb{R}^d$. The Hausdorff distance $d_H$ between $A$ and $B$ is defined by
\begin{equation}\label{hd}
d_H(A,B) 
= \max\left(\sup_{x\in A}\inf_{y\in B}\|x-y\|,\; \sup_{x\in B}\inf_{y\in A}\|x-y\|\right),
\end{equation}
where $\|\cdot\|$ denotes the Euclidean distance on $\mathbb R^d$.
Now let $\mathbb{K}$ be a compact subset of $\mathbb{R}^d$ and let $\varepsilon>0$.
A finite set $\{c_1,\dots,c_P \}\subset \mathbb{K}$ whose elements satisfy $\|c_i-c_j\|>2\varepsilon$ for all $1\le i\neq j\le P$ is called an $\varepsilon$-packing of $\mathbb{K}$. Equivalently, the balls $\bigl(B(c_i,\varepsilon)\bigr)_{1\le i\le P}$ are pairwise disjoint (the (closed) ball, $B(x,\varepsilon)$, for $x\in \mathbb R^d$ is defined by $\{y\in \mathbb R^d,\,\, \|y-x\|\leq \varepsilon\}$). 
An $\varepsilon$-packing of $\mathbb{K}$ is said to be maximal if no additional ball of radius $\varepsilon$, centered in $\mathbb{K}$, can be added without intersecting the others. 
The maximal $\varepsilon$-packing number of $\mathbb{K}$ is defined as
\begin{equation}\label{packingnumber}
P(\mathbb{K},\varepsilon)
= \sup\left\{ N\in\mathbb{N}\;\middle|\; \exists\, x_1,\dots,x_N\in\mathbb{K},\; \|x_i-x_j\|>2\varepsilon \;\; \forall i\neq j \right\}.
\end{equation}

In this work, we investigate the extreme behavior of the minimal distances between a random sample $X_1, \dots, X_n$ of i.i.d.\ random variables in $\mathbb{R}^d$ with compact support $\mathbb{M}$, and carefully chosen deterministic configurations associated with maximal packings of $\mathbb{M}$. In particular, maximal $\varepsilon$-packings of $\mathbb{M}$ play a central role in the study of the Hausdorff distance $d_H(\bbx_n, \mathbb{M})$, where $\bbx_n = \{X_1, \dots, X_n\}$ \cite{Fasy, smale}. {\textcolor{black}{A typical application arises when $X_1$ is uniformly distributed on a compact submanifold $\mathbb{M} \subset \mathbb{R}^d$}}.

Several authors have studied the almost sure convergence to $0$ of $d_H(\bbx_n, \mathbb{M})$ as $n$ tends to infinity under the assumption that the marginal distribution satisfies the $(a,b)$-standard assumption as introduced in Inequality (\ref{kappaepsilon}) (\cite{Ch, CR2004, CR2009, Fasy}). These studies are based on packing sets of the compact support $\mathbb{M}$, following results due to \cite{smale}, and concern both i.i.d.\ observations (\cite{Eddie, Ch, CR2004, CR2009,  Fasy}) and,  dependent ones (\cite{A, kalou, lou1, lou2}). 

A significant motivation for our study comes from the field of statistical inference for geometric and topological data analysis. In particular, \cite{Fasy} developed methods for constructing confidence sets around geometric objects, using distances such as the Hausdorff or bottleneck distance to quantify the discrepancy between data-derived estimates and the true underlying sets. They present several approaches to derive such confidence sets for the Hausdorff distance and, subsequently, for the bottleneck distance between suitable persistence diagrams.

To construct asymptotic confidence sets  using the Hausdorff distance, the almost sure convergence of $d_H(\bbx_n, \mathbb{M})$ is not enough. It is often necessary to derive the asymptotic distribution of the underlying statistics, suitably centered and normalized. Therefore, we are interested in the asymptotic distribution of the Hausdorff distance $d_H(\bbx_n, \mathbb{M})$,  suitably centered and scaled (recall that $\bbx_n=\{X_1,\dots, X_n\}$).

In \cite{Penrose}, the author considers a sequence $X_1, X_2, \ldots$ of independent random points uniformly distributed in a bounded domain $A \subset \mathbb{R}^d$ with smooth boundary. The \emph{coverage threshold} $R_n$ is defined as the smallest radius $r$ such that $A$ is completely covered by the balls of radius $r$ centered at $X_1, \ldots, X_n$. Among other results, limiting distributions for $R_n$ are derived, showing that, after suitable normalization, it converges to a Gumbel law. As noted in \cite{Penrose}, $R_n$ coincides with the Hausdorff distance between the sets $\bbx_n$ and $A$.

In \cite{Penrose2025b}, the authors further investigate the limiting behavior of $R_n$ and extend these results to the setting of a ${\mathcal C}^{2}$ Riemannian manifold with boundary. {\textcolor{black}{It is worth mentioning that their proof is based on the corresponding result for ${\mathcal C}^{2}$ Riemannian manifolds without boundary (see Proposition 2.11 in \cite{Penrose2025b} and the references therein).}} In this more general geometric framework, the limiting laws are again shown to be of Gumbel type.

The Gumbel distribution also appears in the study of the Hausdorff distance between a convex body and the random convex hull generated by interior random points (see \cite{BHB,Prochno}).

Finally, in \cite{Penrose2025a}, assuming that $\partial A$ is either smooth or polygonal, the author considers the random time $\tau_A$ required for $A$ to be fully covered by a spatial birth--growth process evolving in $A$. The paper establishes that the limiting distribution of this covering time is again of Gumbel type.

The objective of the present paper is closely related to that of \cite{Penrose} and \cite{Penrose2025b}. However, in contrast to these works, {\textcolor{black}{we focus on the asymptotic behavior of maximal packings,}} and we do not assume that the support $\mathbb{M}$ is a domain with a smooth boundary. Instead, we impose a different regularity condition on $\mathbb M$, formulated implicitly in Definition~\ref{order}. Roughly speaking, this condition requires that
\[
\mathbb{P}(X\in B(x,\varepsilon))
\]
behaves as $a\varepsilon^{b}$ for small positive $\varepsilon$ and some constants $a>0$ and $b>0$, with an error term of order $\varepsilon^{b+\alpha}$ for some $\alpha>0$ and for all $x\in \mathbb{M}$. 
Importantly, this approximation holds uniformly with respect to $x\in\mathbb{M}$; in particular, the leading constant $a$ does not depend on $x$.

Such a property does not hold in general for domains with boundary. Indeed, if $X$ is uniformly distributed on a compact domain $\mathbb{M} \subset \mathbb{R}^d$ with nonempty boundary (for instance a closed ball, a cube, or more generally a $C^2$ domain), then
\[
\mathbb{P}(X \in B(x,\varepsilon))
= \frac{\mathrm{Vol}(B(x,\varepsilon)\cap \mathbb{M})}{\mathrm{Vol}(\mathbb{M})}.
\]
For points $x$ in the interior of $\mathbb{M}$, the quantity
\[
\mathrm{Vol}(B(x,\varepsilon)\cap \mathbb{M})
\]
coincides with the volume of the full ball when $\varepsilon$ is sufficiently small. 
However, when $x$ is close to the boundary, only a fraction of the ball is contained in $\mathbb{M}$. 
Consequently, the leading constant in the expansion of $\mathbb{P}(X \in B(x,\varepsilon))$ depends on the location of $x$, which prevents the estimate from holding uniformly over $\mathbb{M}$.

The assumption required in Definition \ref{order} is, however, satisfied when $\mathbb{M}$ is a compact $C^2$ manifold without boundary embedded in $\mathbb{R}^d$, and $X$ is uniformly distributed on $\mathbb{M}$. 
In this case, the local geometry is homogeneous and the requirements of Definition \ref{order} hold uniformly for $x\in\mathbb{M}$. 
Typical examples include the sphere, the torus, and more generally any smooth compact submanifold of $\mathbb{R}^d$ without boundary.

In this paper, we provide precise asymptotic Gumbel laws for the minimal distances involved in suitable random packings of embedded compact submanifold without boundary. 
Our main result shows that, under suitable conditions on the underlying distribution and the packing number $P(\mathbb{M},\varepsilon)$, the suitably normalized quantity
\[
\max_{1\le j\le P_n}\min_{1\le i\le n}\|c_j-X_i\|,
\]
where $(c_j)_{1\le j\le P_n}$ are the centers of a maximal packing of $\mathbb{M}$, converges in distribution to a Gumbel law (see Theorem \ref{promain}).
Moreover, we deduce asymptotic bounds on the distribution of the Hausdorff distance between the random point cloud $\bbx_n$ and the compact set $\mathbb{M}$ itself (see Corollary \ref{maintheo}).

Note that when $\mathbb M$ is a compact $C^2$ submanifold (with or without boundary), studying the limiting distribution of the Hausdorff distance between a maximal packing of $\mathbb M$ and a sample of observations has several advantages. 
A maximal packing provides a discrete set of well-separated points that captures the intrinsic geometry of $\mathbb M$ uniformly. 
As a result, the Hausdorff distance to the packing is robust to random sampling fluctuations: adding or removing a few points from the sample has only a limited effect on the distance, unlike the direct estimation of the full support. 
Moreover, the packing naturally encodes the intrinsic dimension and metric structure of $\mathbb M$, so convergence results in Hausdorff distance reflect both the number of sample points and the geometric complexity of the manifold. 
Finally, the limiting distribution provides a rigorous quantitative tool for constructing confidence bands or quantifying uncertainty on geometric features of $\mathbb M$, in a way that is statistically tractable and stable.

The derivation of our results hinges on explicit sequences depending on the parameters of the distribution $a,b$ (appearing in the $(a,b)$-standard assumption, see Definition~\ref{defab} and Definition~\ref{order}) and on geometric characteristics of the set $\mathbb{M}$, related to its maximal packing number $P(\mathbb{M}, \varepsilon)$. 
A key technical tool in our analysis is based on approximating the {\textcolor{black}{inclusion-exclusion principle} using a generalized Newton binomial formula (see the proof of Proposition \ref{main} and of Lemma \ref{propak}).
Another important technical tool that enables the derivation of the scaling and centering constants is the analysis of the Lambert \( W \) function, which is defined as the inverse of the function \( x \mapsto x e^x \) on \( \mathbb{R}^+ \). This function allows for precise asymptotic inversion of expressions that arise in the computation of maximal packing numbers.

Some examples and numerical experiments are presented to illustrate the theoretical findings and to highlight the practical implications of our results.

The paper is organized as follows. In the next section, we introduce the notation and our assumptions, which are of two types. 
The first set of assumptions extends the \((a,b)\)-standard assumption by incorporating upper bounds and an asymptotic probability of uniform sampling near a point on a manifold (Definition~\ref{order}), while the second concerns the maximal $\varepsilon$-packing number of \(\mathbb{M}\), \(P(\mathbb{M}, \varepsilon)\) (see Assumption~\hyperref[assumpP]{${(\mathcal P)}$} below). 

Although upper and lower bounds for \(P(\mathbb{M}, \varepsilon)\) are available in the literature (see~\cite{smale}), what we require here is an asymptotic equivalence for \(\ln(P(\mathbb{M}, \varepsilon))\) as \(\varepsilon \to 0\). We establish such an asymptotic equivalence in Proposition~\ref{theopacking}. This proposition relies on a result of Karp and Pinsky (1989) (see Theorem~\ref{theoK} below), which is also useful for verifying the conditions of Definition~\ref{order}.

The main results are presented in Section~\ref{results}. Some examples with simulations demonstrate that our assumptions can be satisfied, and the simulations are consistent with the theoretical results of Section~\ref{results}. The examples and simulations are the focus of Section~\ref{examples}. All proofs are collected in Section~\ref{proofs}.

Throughout the paper, the notation $a_n = O(b_n)$ (resp.\ $a_n = o(b_n)$ or $a_n \sim b_n$) means that there exists a constant $C > 0$ such that $|a_n| \leq C |b_n|$ for large $n$ (resp.\ $\lim_{n \to \infty} a_n/b_n = 0$ or $\lim_{n \to \infty} a_n/b_n =1$). The notation $[\,\cdot\,]$ denotes the integer part, and \text{cst} refers to a positive constant that may vary from line to line.

\section{Notation and assumptions}
Let $\mathbb{M}$ be a compact subset of $\mathbb R^d$. Let $(X_i)_{1\leq i\leq n}$ be i.i.d. random variables supported  on $\mathbb{M}$ and distributed as a random variable $X$. Since $\bbx_n = \{X_1, \dots, X_n\} \subset \mathbb{M}$, the Hausdorff distance between  $\bbx_n$ and  $\mathbb{M}$, as defined in (\ref{hd}), can be written as
\[
d_H(\bbx_n, \mathbb{M}) = \sup_{x \in \mathbb{M}} \min_{1 \le i \le n} \|x - X_i\|,
\]
where $\|\cdot\|$ denotes the Euclidean distance on $\mathbb R^d$.
We first recall the $(a,b)$-standard assumption on the distribution of $X$ and its extension to the $(a',a,b)$-standard assumption. 
\begin{defi}\label{defab}
Let $X$ be a compactly supported $\mathbb{R}^d$-valued random variable with support $\mathbb{M}$. We say that $X$ satisfies the $(a,b)$-standard assumption if there exist constants $a > 0$, $b > 0$, and $\varepsilon_0 > 0$ such that for all $0 < \varepsilon \leq \varepsilon_0$,
\begin{equation}\label{kappaepsilon}
\inf_{x \in \mathbb{M}} \mathbb{P}(\|X - x\| \leq \varepsilon) \geq a \varepsilon^b.
\end{equation}
We say that $X$ satisfies the $(a',a,b)$-assumption if it satisfies the $(a,b)$-standard assumption and
\begin{equation}\label{ab}
\sup_{x \in \mathbb{M}} \mathbb{P}(\|X - x\| \leq \varepsilon) \leq a' \varepsilon^b,
\end{equation}
for some $a' > 0$ and all $0<\varepsilon \leq\varepsilon_0$.
\end{defi}
The $(a,b)$-standard assumption is well-known in the literature. It has been used, in both the i.i.d. and the dependent\ contexts, in set estimation problems under the Hausdorff distance (\cite{CR2004, CR2009, kalou, lou1}) and in the statistical analysis of persistence diagrams (\cite{Ch, Fasy}). The $(a',a,b)$-assumption was used in \cite{lou2}. 
The $(a',a,b)$-assumption allows to give results on any packing set of $\mathbb{M}$ which is not needed to be a maximal  packing set of $\mathbb{M}$ (see Proposition \ref{main}). The $(a',a,b)$-assumption also provides the following lower and upper bounds for the maximal $\varepsilon$-packing number (introduced in (\ref{packingnumber})), as shown in the following lemma.

\begin{lem}\label{lempacking}
Let $X$ be a random variable, compactly supported on $\mathbb{M}$, and satisfying the $(a',a,b)$-assumption. Then, for any positive $\varepsilon$ small enough,
\[
\frac{1}{a'(2\varepsilon)^b} \le P(\mathbb{M}, \varepsilon) \le  \frac{1}{a \varepsilon^b}.
\]
\end{lem}

\noindent\textbf{Proof of Lemma \ref{lempacking}.}
The proof is immediate since, on the one hand, if $c_1,\ldots, c_p$ is an $\varepsilon$-packing of $\mathbb{M}$ (not necessarily maximal), then under the $(a,b)$-standard assumption,
\begin{equation}\label{P1}
p\, a\, \varepsilon^b 
\ \leq\ p \min_{1\leq i\leq p}\mathbb{P}\bigl(X \in B(c_j,\varepsilon)\bigr)\leq 
\sum_{j=1}^{p} \mathbb{P}\bigl(X \in B(c_j,\varepsilon)\bigr)
\ = \ 
\mathbb{P}\biggl(\bigcup_{j=1}^{p} B(c_j,\varepsilon)\biggr)
\ \leq \ 1,
\end{equation}
so that the upper bound is proved taking $p=P(\mathbb{M}, \varepsilon)$.
On the other hand, letting $p=P(\mathbb{M}, \varepsilon)$, since $c_1,\ldots, c_p$ is a maximal $\varepsilon$-packing of $\mathbb{M}$, then $\mathbb{M} \subset \bigcup_{j=1}^{P(\mathbb{M}, \varepsilon)} B(c_j, 2\varepsilon)$ (see Lemma 5.2 in \cite{smale}). Consequently,
\begin{eqnarray}\label{P2}
&& 1 
\ = \ 
\mathbb{P}\bigl(X \in \mathbb{M}\bigr)
\ \le \ 
\mathbb{P}\biggl(X \in \bigcup_{j=1}^{P(\mathbb{M}, \varepsilon)} B(c_j, 2\varepsilon)\biggr)
\ \le \ 
\sum_{j=1}^{P(\mathbb{M}, \varepsilon)} \mathbb{P}\bigl(X \in B(c_j, 2\varepsilon)\bigr) {\nonumber}\\
&& \le \ 
P(\mathbb{M}, \varepsilon) \max_{1\leq j\leq P(\mathbb{M}, \varepsilon)}\mathbb{P}\bigl(X \in B(c_j, 2\varepsilon)\bigr)
\ \le \ 
P(\mathbb{M}, \varepsilon)\, a'\, (2\varepsilon)^b.
\end{eqnarray}
This completes the proof of Lemma \ref{lempacking}.

\bigskip

To get our desired results, the $(a',a,b)$-assumption is not enough.
We need a more precise decomposition of $\mathbb{P}(X\in B(c_j,\varepsilon))$ when $\varepsilon$ tends $0$, for the packing set $(c_j)_j$. 
For this reason, we supposed for the main result that $X$ is uniformly distributed on $\mathbb{M}$ and $\mathbb{M}$ is sufficiently regular in order to satisfy the requirement of the definition below.
\begin{defi}\label{order}
Let $X$ be a compactly supported $\mathbb{R}^d$-valued random variable uniformly distributed on $\mathbb{M}$. We say that $X$ satisfies the $(a,b)$-standard assumption with rate $\alpha$ if
there exists $\varepsilon_0$ and  constants $\alpha>0$, $a>0$, $b>0$ and  $d>0$ such that for any $\varepsilon \in (0,\varepsilon_0)$ and any $x\in \mathbb{M}$, it holds
$$
 \left|\mathbb{P} (X\in B(x, \varepsilon)) - a \varepsilon^b \right| \le d \varepsilon^{b+\alpha}.
$$
\end{defi}
When $X$ is uniformly distributed, the evaluation of $\mathbb{P}(X \in B(x, \varepsilon))$ is related to the volume of small balls $B(x, \varepsilon) \cap \mathbb{M}$. 
The following theorem (see \cite{karp-pinsky}) provides an estimate of this volume, 
which, in turn, allows us to provide many examples of compact smooth submanifolds without boundary $\mathbb{M} \subset \mathbb{R}^d$,
for which the requirements of Definition \ref{order} are satisfied.

\begin{theo}[Karp-Pinsky, 1989] \label{theoK}
Let $M \subset \mathbb{R}^d$ be a smooth $p$-dimensional submanifold of $\mathbb R^d$ and let $x \in M$.  Then, as $\varepsilon \to 0$, the volume of the intersection of $M$ with the extrinsic Euclidean ball $B(x,\varepsilon)$ satisfies
{\textcolor{black}{
\[
\operatorname{Vol}\big(M \cap B(x,\varepsilon)\big)
=
\frac{\sigma_{p-1}}{p}\,\varepsilon^p
+
\frac{\sigma_{p-1}}{8p(p+2)}
\left(
2\|B_x\|^2
-
\|H_x\|^2
\right)
\varepsilon^{p+2}
+
O(\varepsilon^{p+3}),
\]
where}}

\begin{itemize} 
\item $\sigma_{p-1}$ denotes the surface measure of the unit sphere in
$\mathbb{R}^p$, that is \[
\sigma_{p-1} = \operatorname{Vol}_{p-1}(S^{p-1}) = \frac{2 \pi^{p/2}}{\Gamma(p/2)}
\]
\item $B_x$ is the second fundamental form of $M$ at $x$,
\item $H_x = \operatorname{trace}(B_x)$ is the mean curvature normal vector,
\item $\|B_x\|$ is the Hilbert--Schmidt norm of $B_x$, defined by
\[
\|B_x\|^2
=
\sum_{i,j=1}^p
\|B_x(e_i,e_j)\|^2,
\]

\item $\|H_x\|$ is the Euclidean norm of $H_x$,

\item and $\{e_i\}_{i=1}^p$ is any orthonormal basis of $T_x M$. 
\end{itemize} 
The remainder term $O(\varepsilon^{p+3})$ is uniform for $x$ in compact subsets
of $M$. 
\end{theo}
Thanks to this theorem, we provide examples in Section \ref{examples} for which the requirements of Definition \ref{order} are satisfied.

Our last assumption concerns the maximal $\varepsilon$-packing number $P(\mathbb{M}, \varepsilon)$, as defined in (\ref{packingnumber}), which corresponds to the maximal number of disjoint balls of radius $\varepsilon$ centered in $\mathbb{M}$. 
Unfortunately, upper and lower bounds for $P(\mathbb{M}, \varepsilon)$, as stated in Lemma \ref{lempacking}, are not sufficient for our purposes. 
We need an exact asymptotic order for $\ln(P(\mathbb{M}, \varepsilon))$ as $\varepsilon$ tends to $0$. We formulate this requirement in the following assumption.

\begin{enumerate}
\item[${(\mathcal P)}$]\refstepcounter{assump}\label{assumpP} 
Suppose that there exists a constant \( c(\mathbb{M}) \), depending on $\mathbb{M}$, such that
{\textcolor{black}{\[
\frac{\ln\left(P(\mathbb{M}, \varepsilon)\right)}{
\ln(c(\mathbb{M})\, \varepsilon^{-b}) + o(1)}\to 1
\quad \text{as } \varepsilon \to 0.
\]}}
\end{enumerate}
We give in Section \ref{examples} (see Proposition \ref{theopacking}) sufficient conditions on $\mathbb{M}$ under which Assumption~\hyperref[assumpP]{${(\mathcal P)}$} is satisfied.
We also provide practical examples.

\section{Results}\label{results}
{\textcolor{black}{
Before stating our main Theorem \ref{promain} and its corollary (Corollary \ref{maintheo}), 
we present three preliminary results: 
Lemma~\ref{propak}, Proposition~\ref{main}, 
and Corollary~\ref{main2}. 
The purpose of these results is to control
\[
\mathbb{P}\!\left(
\max_{1 \le j \le P_n}
\min_{1 \le i \le n}
\| X_i - c_j \|
\le u_n
\right),
\]
(see the statements for further details on the notation). 
From one result to the next, the assumptions are progressively strengthened.
\\
We begin with the following key lemma, which requires only the $(a',a,b)$-assumption. 
In particular, the i.i.d.\ random variables are not assumed to be uniformly distributed, 
and the considered packing set need not be maximal.
}}
\begin{lem}\label{propak}
Let $X_1, \ldots, X_n$ be i.i.d.\ random variables distributed as $X$, supported on the compact set $\mathbb{M}$. 
Suppose that the $(a',a,b)$-assumption is satisfied. Let $(u_n)_n$ be a sequence {\textcolor{black}{of strictly positive numbers}} tending to $0$
as $n \to \infty$, for which there exists a positive constant \(K\) such that, for any positive integer \(n\),
\begin{equation}\label{conditionun}
\left| a\, n\, u_n^b 
- \ln(n) 
+ \ln\bigl(\ln(n)\bigr) \right| 
\le K.
\end{equation}
Let $c_1,\ldots, c_{P_n}$ be a $u_n$-packing of $\mathbb{M}$ i.e., for any \(1 \le i \neq j \le P_n\),
\begin{equation}\label{intersectionvide}
c_i \in \mathbb{M},
\quad
B\bigl(c_i, u_n\bigr) \cap B\bigl(c_j, u_n\bigr) = \emptyset.
\end{equation}
Then,
\[
\lim_{n \to \infty} \left| \mathbb{P}\left( \max_{1 \le j \le P_n} \min_{1 \le i \le n} \|X_i - c_j\| \le u_n \right) 
- \prod_{j=1}^{P_n} \Bigl( 1 - \exp\bigl( -n \, \mathbb{P}\left(\|X_1 - c_j\| \le u_n\right) \bigr) \Bigr) \right| = 0.
\]
\end{lem}

\noindent\textbf{Remark.} \textcolor{black}{
As will be seen in the proof of Lemma~\ref{propak}, 
Assumption~\eqref{intersectionvide} can be replaced by the weaker condition
\[
\mathbb{P}\bigl(X \in B(c_i, u_n) \cap B(c_j, u_n)\bigr) = 0,
\]
for all $i \neq j$, and the conclusion of the lemma still holds. 
In other words, the balls $B(c_i,u_n)$ and $B(c_j,u_n)$ may intersect, 
but the probability that $X$ belongs to their intersection is zero.}
\\
If we strengthen the \((a',a,b)\)-assumption to the condition given in Definition~\ref{order}, we obtain a sharper limit result than the one stated in Lemma~\ref{propak}. This is the purpose of Proposition~\ref{main}. 

Note that Proposition~\ref{main} also involves a \(u_n\)-packing, which is not required to be maximal.

\begin{pro}\label{main}
Let \(X_1, \ldots, X_n\) be i.i.d. random variables uniformly distributed on the compact set $\mathbb{M}$ of ${\mathbb R^d}$. Suppose that the requirement of Definition \ref{order} is satisfied. Let \((u_n)\) be a sequence {\textcolor{black}{of positive numbers}} tending to \(0\) as \(n\) tends to infinity, for which Condition (\ref{conditionun}) is satisfied.
\
Let \(c_1, \ldots, c_{P_n}\) be a \(u_n\)-packing of the set \(\mathbb{M}\). 
Then,
\[
\lim_{n \rightarrow \infty}
\left|
\mathbb{P}\left(
\max_{1 \le j \le P_n}\,
\min_{1 \le i \le n}
\bigl\| X_i - c_j \bigr\|
\le u_n
\right)
-
\exp\left(
- P_n \exp\bigl(- a\, n\, u_n^b\bigr)
\right)
\right|
= 0.
\]
\end{pro}
We now suppose that Assumption~~\hyperref[assumpP]{${(\mathcal P)}$}  holds. In this case, we still obtain a sharper limit result.

\begin{coro}\label{main2}
Let \(X_1, \ldots, X_n\) be i.i.d. random variables uniformly distributed on the compact set $\mathbb{M}$ of ${\mathbb R^d}$. Suppose that the requirement of Definition \ref{order} is satisfied.
 Suppose also that Assumption ~\hyperref[assumpP]{${(\mathcal P)}$}  is satisfied.
Let \(u \in \mathbb{R}\) be fixed and define a sequence \(\bigl(U_n(u)\bigr)_n\) satisfying
\begin{equation}\label{lambert}
\lim_{n \rightarrow \infty} 
\left| 
n\,a\, U_n^b(u)
-
W\bigl(n\, a\, c(\mathbb{M})\, \exp(u)\bigr) 
\right|
= 0,
\end{equation}
where \(W\) is the Lambert function and $c(\mathbb{M})$ is the constant appearing in Assumption ~\hyperref[assumpP]{${(\mathcal P)}$}.
Let \(c_1,\dots, c_{P_n}\) be a \(U_n(u)\)-maximal packing (satisfying Assumption ~\hyperref[assumpP]{${(\mathcal P)}$}).
Then,
\[
\lim_{n \rightarrow \infty} 
\mathbb{P}\left(
\max_{1 \le j \le P_n}\,
\min_{1 \le i \le n}
\|X_i - c_j\| \le U_n(u)
\right)
=
\exp\bigl(- \exp(-u)\bigr).
\]
\end{coro}

We are now ready to state our main theorem and its corollary.

\begin{theo}\label{promain}
Let \(X_1, \ldots, X_n\) be i.i.d. random variables, uniformly distributed on the compact subset $\mathbb{M}$ of $\mathbb{R}^d$.
Suppose that the requirement of Definition \ref{order}  is satisfied. Suppose also that Assumption ~\hyperref[assumpP]{${(\mathcal P)}$}  is satisfied.
Define the two sequences
\[
\alpha_n = \frac{1}{b\, a^{1/b}} \frac{(\ln n)^{1/b - 1}}{n^{1/b}},
\]
and
\[
\beta_n = \frac{1}{a^{1/b}} \frac{(\ln n)^{1/b}}{n^{1/b}}
\left( 
1 + \frac{1}{b\, \ln(n)} \bigl( \ln\bigl(a\, c(\mathbb{M})\bigr) - \ln(\ln(n)) \bigr)
\right).
\]
Let \(u \in \mathbb{R}\) be fixed. {\textcolor{black}{Let $n$ be large enough so that $\alpha_n u + \beta_n > 0$.}} Let \(c_1, \ldots, c_{P_n}\) be a \((\alpha_nu+\beta_n)\)-maximal packing of  \(\mathbb{M}\). 
Then,
$$
\lim_{n \rightarrow \infty} 
\mathbb{P}\left(
\bigl\{c_1,\cdots, c_{P_n}\bigr\} \subset \bigcup_{i=1}^n B(X_i, \alpha_n\,u+ \beta_n)\right)
=
\exp\bigl(- \exp(-u)\bigr).
$$
\end{theo}
\noindent {\textcolor{black}{An immediate consequence of the above theorem is the following corollary, 
which provides, in particular, upper and lower bounds on the cumulative 
distribution function of the suitably normalized Hausdorff distance 
\( d_H(\mathbb{X}_n,\mathbb{M}) \), that is, bounds for
$
\mathbb{P}\bigl(
d_H(\mathbb{X}_n,\mathbb{M})
\le \alpha_n u + \beta_n
\bigr).
$}}

\begin{coro}\label{maintheo}
Suppose that all the requirements of Theorem \ref{promain} are satisfied.
Then, for any \(u \in \mathbb{R}\),
\[
\exp\bigl(-\exp(-u)\bigr)
\;\le\; 
\liminf_{n \rightarrow \infty} 
\mathbb{P}\Bigl(\bigcup_{j=1}^{P_n} B(c_j, \alpha_n\, u + \beta_n ) \subset \bigcup_{i=1}^n B(X_i, 2\,\alpha_n\, u + 2\, \beta_n )\Bigr),
\]
\[
\limsup_{n \rightarrow \infty} \textcolor{black}{\mathbb{P}\bigl(
d_H(\mathbb{X}_n,\mathbb{M})
\le \alpha_n u + \beta_n
\bigr)}
\;\le\; \exp\bigl(-\exp(-u)\bigr),
\]

and
\begin{equation} 
\label{promainoct}
\exp\bigl(-\exp(-u)\bigr)
\;\le\; 
\liminf_{n \rightarrow \infty} {\textcolor{black}{\mathbb{P}\bigl(
d_H(\mathbb{X}_n,\mathbb{M})
\le 3\alpha_n u + 3\beta_n
\bigr).}}
\end{equation}
\end{coro}
\noindent {\bf{Remark.}} {\textcolor{black}{
Note that if one is only interested in the convergence in distribution of the Hausdorff distance $d_H(\mathbb{X}_n,\mathbb{M})$ (when $\mathbb M$ is a compact $d$-dimensional Riemannian manifold
without boundary), then Proposition 2.11 in \cite{Penrose2025b} (established in a general setting) and the references therein already addresses this question, with centering and normalization constants that are asymptotically equivalent (but not equal) to $\alpha_n$ and $\beta_n$. We also note that, in this proposition, the exponent $b$ corresponds to, $d$, the dimension of the submanifold $\mathbb{M}$ and is therefore necessarily an integer. In our setting, however, we only require that $b>0$, not necessarily an integer, which could allow for examples related to fractal sets, although we do not pursue this direction here. Furthermore, the result of our Theorem \ref{promain} on maximal packings cannot be deduced from Proposition 2.11 in \cite{Penrose2025b}. Finally, the techniques used in the proofs are substantially different.
}}

\section{Examples and simulations}\label{examples}

\noindent {\textcolor{black}{The objective of this section is to provide examples, together with simulations, illustrating the conclusions of Theorem \ref{promain} and its corollary. Since this theorem relies, among other things, on Assumption~\hyperref[assumpP]{${(\mathcal P)}$}, we state the following proposition, which provides sufficient conditions for this Assumption~\hyperref[assumpP]{${(\mathcal P)}$}. 
\begin{pro}\label{theopacking}
Assume that the assumptions of Theorem~\ref{theoK} are satisfied. 
Assume moreover that the quantity
\[
\bigl| 2\|B_x\|^2 - \|H_x\|^2 \bigr|
\]
is uniformly bounded with respect to $x \in \mathbb{M}$. 
Then, as $\varepsilon \to 0$,
\[
\frac{\ln P(\mathbb{M}, \varepsilon)}
{\ln\!\left( \left(\frac{1}{\varepsilon}\right)^p 
\frac{p\,\mathrm{Vol}(\mathbb{M})}{\sigma_{p-1}} \right) + o(1)}
\longrightarrow 1.
\]
\end{pro}
}}
\noindent {\textcolor{black}{
We prove Proposition \ref{theopacking} later in Subsection \ref{proofPro}, and then continue the study of examples with simulations. We now have all the ingredients needed for this study. For each example, we verify that the assumptions of our Theorem~\ref{promain} are satisfied; more precisely, the $(a,b)$-standard assumption with rate $\alpha$ (as described in Definition~\ref{order}), as well as Assumption~\hyperref[assumpP]{${(\mathcal P)}$}. We also explicitly provide the expressions of $a$, $b$, and $c(\mathbb{M})$ for the three studied examples. Note that knowing these values of $a$, $b$, and $c(\mathbb{M})$ is sufficient to derive the centering and normalizing sequences $(\alpha_n)_n$ and $(\beta_n)_n$ of Theorem~\ref{promain}. }}

{\textcolor{black}{
The three examples studied concern the circle, the sphere, and the torus. For each example, we verify the assumptions of Theorem~\ref{promain} based on Theorem~\ref{theoK}. In particular, we illustrate the convergence toward the Gumbel law established in Theorem~\ref{promain}, as well as the asymptotic confidence bands derived in Corollary~\ref{maintheo}.  
The illustration of the convergence to the Gumbel law is based on the following algorithm that outlines the different steps involved in these simulations.
}}


\begin{algorithm}
\caption{Simulation of the Coverage Probability on $\mathbb M$}
\begin{algorithmic}[1]

\State \textbf{Input:} Sample size $n$, number of trials $T$, list of values $u_1, \ldots, u_K$
\State \textbf{Constants:} $a$, $b$, $c(\mathbb{M})$

\For{$k = 1$ to $K$}
    \State $u \gets u_k$
    \State Compute $\alpha_n = \dfrac{1}{b\, a^{1/b}} \dfrac{(\ln n)^{1/b - 1}}{n^{1/b}}$
    \State Compute $\beta_n = \dfrac{1}{a^{1/b}} \dfrac{(\ln n)^{1/b}}{n^{1/b}} \left(1 + \dfrac{1}{b\, \ln(n)} \left( \ln(ac(\mathbb{M})) - \ln(\ln(n)) \right)\right)$
    \State $r \gets \alpha_n \cdot u + \beta_n$
    \State Generate a \textbf{$r$-maximal packing} $\{c_1, \ldots, c_P\}$ of $\mathbb M$.
    \State $C \gets 0$ 

    \For{$t = 1$ to $T$}
        \State Generate $x_1, \ldots, x_n$ uniformly on $\mathbb M$
        \State \textbf{For each} center $c_j$, check if there exists a point $x_i$  such that \newline $\|x_i - c_j\| \leq r$
        \If{all $c_j$ are covered}
            \State $C \gets C + 1$
        \EndIf
    \EndFor

    \State Estimate empirical probability: $p_k^{(sim)} \gets \dfrac{C}{T}$
    \State Compute theoretical value: $p_k^{(theo)} \gets \exp(-\exp(-u))$
\EndFor

\State \textbf{Output:} Plot of $(u_k, p_k^{(sim)})$ vs. $(u_k, p_k^{(theo)})$

\end{algorithmic}
\end{algorithm}

%

\subsection{The unit circle on $\mathbb R^2$}
Consider the compact set
\[
\mathbb{M} = \mathbb{S}^1 
= \bigl\{ (x,y) \in \mathbb{R}^2 : x^2 + y^2 = 1 \bigr\}
\]
equipped with the \emph{arc-length measure}.  Suppose that $X$ is distributed as the uniform law on the circle $\mathbb{S}^1$.
Let $x \in \mathbb{S}^1$. For small $\varepsilon > 0$, we have 
\begin{eqnarray*}
&& \mathbb{P}\bigl(X\in B(x, \varepsilon)\bigr)
=
\frac{\operatorname{length}\bigl(B(x,\varepsilon)\cap \mathbb{S}^1 \bigr)}{2\pi}.
\end{eqnarray*}
We apply Theorem \ref{theoK} with $p=1$
The principal curvature of $\mathbb{S}^1$ is 
$
\kappa = 1.
$
We also have, for this example, 
\[
\|B_x\|^2 = \kappa^2, 
\qquad
\|H_x\|^2 = \kappa^2.
\]
Hence
\[
2\|B_x\|^2 - \|H_x\|^2 = 1.
\]
Moreover, the surface measure of the $0$-dimensional sphere is
\[
\sigma_0 = 2.
\]
The volume (i.e., the length) of the intersection of the circle with the extrinsic ball satisfies then, thanks to Theorem \ref{theoK}, for $\varepsilon \to 0$, 
\[
\operatorname{length}\bigl(B(x,\varepsilon)\cap \mathbb{S}^1 \bigr) = 2\varepsilon + \frac{\varepsilon^3}{12} + O(\varepsilon^4),
\]
the reminder term $O(\varepsilon^4)$ is uniform for $x\in \mathbb{S}^1$. 
Consequently,

\begin{eqnarray*}
&& \mathbb{P}\bigl(X\in B(x, \varepsilon)\cap \mathbb{S}^1\bigr)
= \frac{\varepsilon}{\pi}  + \frac{\varepsilon^2}{24\pi} + O(\varepsilon^4).
\end{eqnarray*}
Hence, Definition \ref{order} is satisfied with $$a=\frac{1}{\pi},\,\, b=1,\,\, \alpha=1.$$  
\\
Since  $2\|B_x\|^2 - \|H_x\|^2 = 1$, Proposition \ref{theopacking} applies and ensures Assumption
~\hyperref[assumpP]{${(\mathcal P)}$} with $b=p=1$ and 
$$
c(\mathbb{S}^1)= \frac{\,\operatorname{length}(\mathbb{S}^1)}{\sigma_{0}}=  \frac{2\pi}{2}= \pi.
$$

\subsubsection*{Convergence to the Gumbel distribution}
The purpose now is to perform simulations of $n$ independent realizations of a random variable uniformly distributed on the unit circle $\mathbb{S}^1$. The goal is to illustrate the convergence of the coverage probability towards a Gumbel distribution, as stated in Theorem~\ref{promain}. 
\begin{figure}[htb]
\begin{center}
\includegraphics[width=12cm]{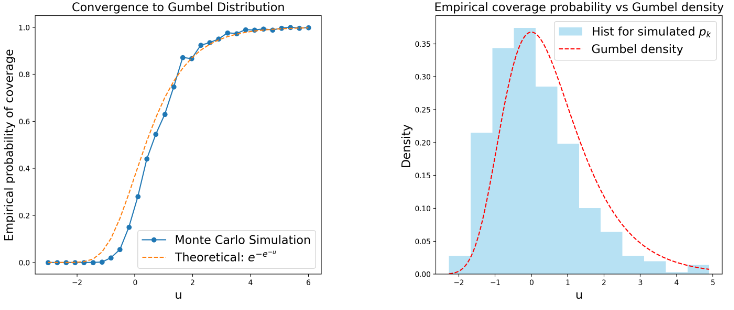}
\vspace{-0.3cm}
\caption{{\small {Illustration of the convergence to the Gumbel distribution: the case of the unit circle.}}}
\end{center}
\end{figure}

\subsubsection*{Confidence sets} 
The purpose now is to illustrate confidence sets for a maximal packing set, as well as for the set $\mathbb{M}$.

\begin{figure}[htb]
\begin{center}
\includegraphics[width=9.5cm]{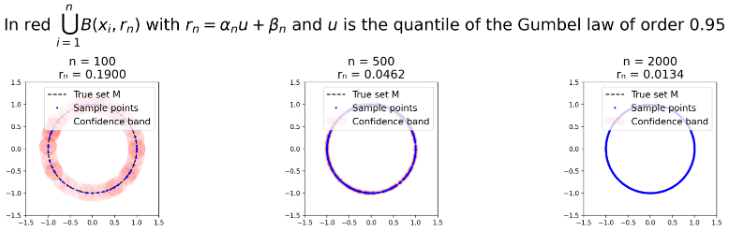}
\caption{{ Asymtotic confidence set for a maximal $r_n$-packing set (Theorem \ref{promain})}}
\end{center}
\vspace{0.5cm}
\begin{center}
\includegraphics[width=9.5cm]{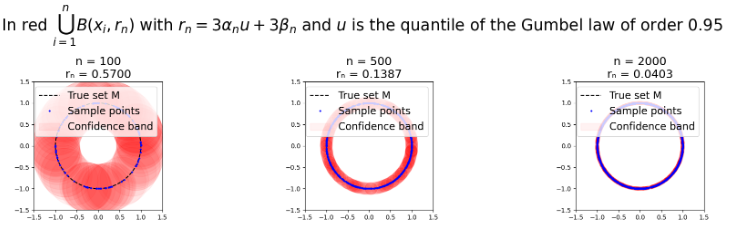}
\caption{Asymtotic confidence set for the support $\mathbb{M}$ (the limit in (\ref{promainoct}))}
\end{center}
\end{figure}

\newpage

\subsection{The Sphere \texorpdfstring{$\mathbb{S}^2$}{S²} in \texorpdfstring{$\mathbb{R}^3$}{R³}}
The unit sphere is given by:
\[
\mathbb{S}^2 = \left\{ (x, y, z) \in \mathbb{R}^3 : x^2 + y^2 + z^2 = 1 \right\},
\]
it's a compact $2$-dimensional  manifold without boundary. We apply Theorem \ref{theoK} with $p=2$.
The principal curvatures
are $\kappa_1 = \kappa_2 = 1$, hence
\[
\|B_x\|^2 = 2,
\qquad
\|H_x\|^2 = 4,
\]
and therefore
\[
2\|B_x\|^2 - \|H_x\|^2 = 0.
\]
Moreover, the surface measure of the $1$-dimensional unit sphere is
\[
\sigma_1 = 2\pi.
\]
It follows from Theorem \ref{theoK} that, for any $x \in S^2$, as $\varepsilon \to 0$,
\[
\operatorname{Vol}(\mathbb{S}^2 \cap B(x,\varepsilon))
=
\pi \varepsilon^2
+
O(\varepsilon^5),
\]
$\operatorname{Vol}$ means here the surface and the reminder term $O(\varepsilon^5)$ is uniform for $x \in \mathbb{S}^2$. 
Consequently, for small positive $\varepsilon$, $x\in \mathbb{S}^2$ and $X$ a random variable uniformly distributed on $\mathbb{S}^2$, it holds
\begin{eqnarray*}
&& \mathbb{P}(X \in \mathbb{S}^2 \cap B(x,\varepsilon)) =  \frac{\operatorname{Vol}(B(x,\varepsilon) \cap \mathbb{S}^2)}{4 \pi} = \frac{\varepsilon^2}{4} +  O(\varepsilon^5).
\end{eqnarray*}
Hence, Definition \ref{order} is satisfied with
$$
a=\frac{1}{4},\,\, b=2,\,\, \alpha=3.
$$
Since $2\|B_x\|^2 - \|H_x\|^2 = 0$, Proposition \ref{theopacking} applies and ensures Assumption 
~\hyperref[assumpP]{${(\mathcal P)}$} with $b=p=2$ and 
$$
c(\mathbb{S}^2)= \frac{p\,\mathrm{Vol}(\mathbb{S}^1)}{\sigma_{p-1}}=  \frac{8\pi}{2\pi}= 4.
$$
\subsubsection*{Convergence to the Gumbel distribution.}
The goal is to illustrate the convergence of the coverage probability towards a Gumbel distribution, as stated in Theorem~\ref{promain}. 
\begin{figure}[htb]
\begin{center}
\includegraphics[width=12cm]{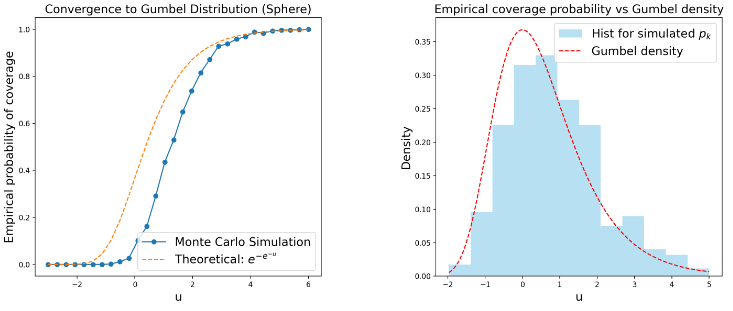}
\vspace{-0.3cm}
\caption{{\small {Illustration of the convergence to the Gumbel distribution: the case of the unit Sphere.}}}
\end{center}
\end{figure}
\subsubsection*{Confidence sets.}
The purpose is to illustrate the conclusion of Corollary \ref{maintheo}.
\begin{figure}[htb]
\begin{center}
\includegraphics[width=10cm]{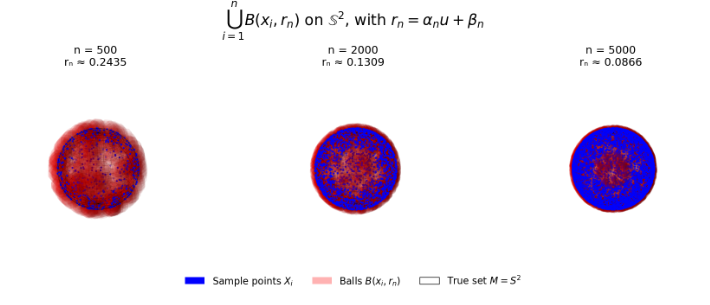}
\end{center}
\vspace{0.25cm}
\begin{center}
\includegraphics[width=10cm]{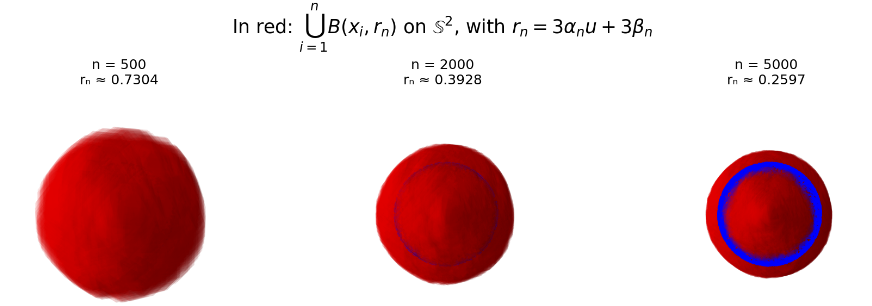}
\caption{Asymtotic confidence set for a maximal packing set and for the support (Theorem \ref{promain} and of the limit in (\ref{promainoct})), $u$ is still the quantile of the Gumbel law of order $0.95$.}
\end{center}
\end{figure}

\newpage

\subsection{The torus $T(R,r)$ in $\mathbb{R}^3$}
Let 
\[
T(R,r) = \left\{ (x_0,y_0,z_0) \in \mathbb{R}^3 : (\sqrt{x_0^2+y_0^2}-R)^2 + z_0^2 = r^2 \right\}
\]
be the standard torus with major radius $R > r > 0$.
\begin{figure}[htb]
\begin{center}
\includegraphics[width=4cm]{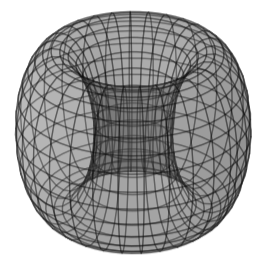}
\caption{The torus T(2,1)}
\end{center}
\end{figure}

Here, the dimension of the submanifold is $p=2$, and the ambient space dimension is $d=3$.
Parametrizing by $(\theta,\phi) \in [0,2\pi)^2$,
\[
\begin{cases}
x_0 = (R + r \cos\theta) \cos\phi, \\
y_0 = (R + r \cos\theta) \sin\phi, \\
z_0 = r \sin\theta,
\end{cases}
\]
the principal curvatures are
\[
\kappa_1 = \frac{\cos\theta}{R + r \cos\theta}, 
\qquad
\kappa_2 = \frac{1}{r}.
\]
Hence,
\[
\|B_x\|^2 = \kappa_1^2 + \kappa_2^2, 
\qquad
\|H_x\|^2 = (\kappa_1 + \kappa_2)^2, 
\qquad
2\|B_x\|^2 - \|H_x\|^2 = (\kappa_1 - \kappa_2)^2.
\]
Finally, the surface measure of the unit sphere in $\mathbb{R}^2$ is
\[
\sigma_1 = 2\pi.
\]
Then, Theorem \ref{theoK} gives, as $\varepsilon \to 0$, 
\[
\operatorname{Vol}\big(T(R,r) \cap B(x,\varepsilon)\big)
=
\pi \, \varepsilon^2
+
\frac{\pi}{32} \, (\kappa_1 - \kappa_2)^2 \, \varepsilon^4
+
O(\varepsilon^5),
\]
$O(\varepsilon^5)$ and $(\kappa_1 - \kappa_2)^2$ are uniformly bounded over $\theta$ and $\phi$ (and then over $x_0,y_0,z_0$). The requirements of Definition \ref{order} are then satisfied (noting that the surface of $T(R,r)$ is equal to $4\pi^2Rr$) with,
$$
a= \frac{\pi}{4\pi^2Rr}=\frac{1}{4\pi Rr},\,\,\, b=2,\,\,\, \alpha=2.
$$
Since $2\|B_x\|^2 - \|H_x\|^2$ is uniformly bounded over $x$ Proposition \ref{theopacking} applies 
and ensures Assumption 
~\hyperref[assumpP]{${(\mathcal P)}$} with $b=p=2$ and 
$$
c(T(R,r))= \frac{p\,\mathrm{Vol}(T(R,r))}{\sigma_{p-1}}=  \frac{8\pi^2 Rr}{2\pi}= 4 \pi Rr.
$$
\subsubsection*{Convergence to the Gumbel distribution.}
The purpose is to illustrate the convergence to the Gumbel distribution.
\begin{figure}[htb]
\begin{center}
\includegraphics[width=12cm]{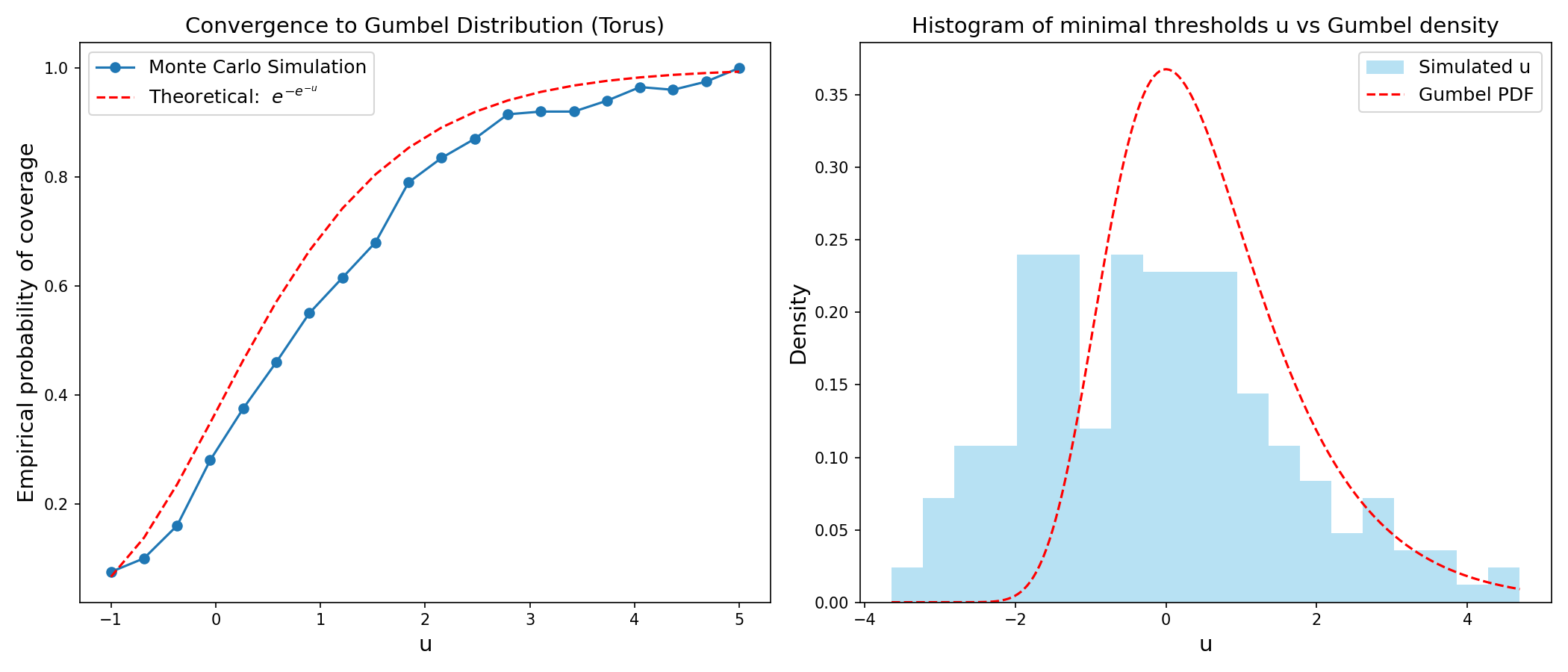}
\vspace{-0.3cm}
\caption{{\small {Illustration of the convergence to the Gumbel distribution: the case of the torus T(2,1). The convergence of the histogram toward the Gumbel density is imperfect, due to finite-sample size and random packing effects.}}}
\end{center}
\end{figure}

\subsubsection*{Confidence sets.}
We now illustrate the conclusion of Corollary \ref{maintheo} for the torus $T(2,1)$.
\begin{figure}[htb]
\begin{center}
\includegraphics[width=12cm]{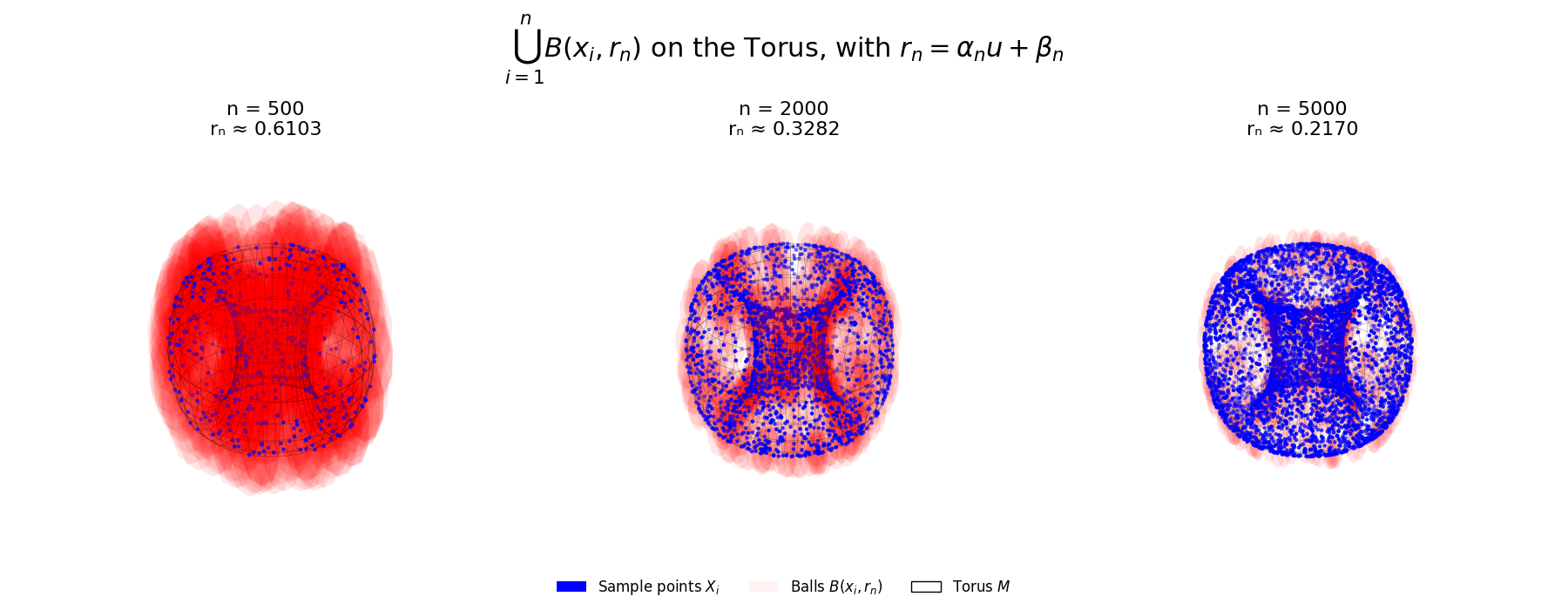}
\caption{Asymtotic confidence set for a maximal packing set (Theorem \ref{promain}), $u$ is still the quantile of the Gumbel law of order $0.95$.}
\end{center}
\end{figure}

\section{Proofs}\label{proofs}

We  recall first some bounds that will be needed in several places.
If the sequence \((u_n)_n\) satisfies \eqref{conditionun}, then clearly,
\begin{equation}\label{bound1}
\lim_{n \rightarrow \infty} \frac{a\, n\, u_n^b}{\ln n} = 1,
\quad
\exp\bigl(- a\, n\, u_n^b \bigr) 
\le e^K \frac{\ln n}{n},
\quad
n\, u_n^{2b} = O\left( \frac{\ln^2(n)}{n} \right).
\end{equation}
Now, let \(c_1, \ldots, c_{P_n}\) be a \(u_n\)-packing of \(\mathbb M\), not necessarily maximal. We need an upper bound for \(P_n\), where $(u_n)$ is the sequence satisfying Condition (\ref{conditionun}). From Lemma \ref{lempacking} and its proof, we have
\[
a\, u_n^b\, P_n
\;\le\; 1.
\]
Therefore, by Condition \eqref{conditionun}, we obtain for \(n\) large enough,
\begin{equation}\label{bound2}
P_n 
\;\le\; 
\frac{1}{au_n^{b}}
\;\le\;
\frac{n}{\ln(n) - \ln(\ln(n)) - K}
\;\le\; 
\text{cst} \cdot \frac{n}{\ln(n)}.
\end{equation}

Bounds \eqref{bound1} and \eqref{bound2} will be used in several places throughout this paper.
\subsection{Proof of Lemma \ref{propak}}
{\textcolor{black}{
Since the purpose of this lemma is to control the difference
\[
\mathbb{P}\!\left(
\max_{1 \le j \le P_n}
\min_{1 \le i \le n}
\|X_i - c_j\| \le u_n
\right)
-
\prod_{j=1}^{P_n}
\left(1 - e^{-a_j}\right),
\]
where $a_j = n\, \mathbb{P}(\|X_1 - c_j\| \le u_n)$, 
we apply the inclusion--exclusion principle to expand
\[
\mathbb{P}\!\left(
\max_{1 \le j \le P_n}
\min_{1 \le i \le n}
\|X_i - c_j\| \le u_n
\right)
\]
and a generalization of the binomial formula to expand
\[
\prod_{j=1}^{P_n}
\left(1 - e^{-a_j}\right).
\]
We then compare the terms of the two expansions one by one.
}}
For this define, for $1\leq j\leq P_n$, the event 
$$
A_j=\left\{\min_{1\leq i\leq n}\|X_i-c_j\|> u_n\right\}.
$$
Clearly,
$$
\mathbb{P}\left(\max_{1\leq j\leq P_n}\min_{1\leq i\leq n}\|X_i-c_j\|> u_n\right) 
= \mathbb{P}\left(A_1\cup\cdots \cup A_{P_n}\right).
$$
Hence, on the one hand, using the inclusion-exclusion principle,
\begin{eqnarray}\label{poincare}
&& \mathbb{P}\left(\max_{1\leq j\leq P_n}\min_{1\leq i\leq n}\|X_i-c_j\|\leq u_n\right) \nonumber \\
&& = 1- \mathbb{P}(A_1\cup\cdots \cup A_{P_n}) \nonumber \\
&& = 1 - \sum_{j=1}^{P_n} \mathbb{P}(A_j)
      - \sum_{k=2}^{P_n} (-1)^{k+1} 
        \sum_{1\leq i_1<i_2<\cdots<i_k\leq P_n} 
        \mathbb{P}\bigl(A_{i_1}\cap \cdots \cap A_{i_k}\bigr).
\end{eqnarray}

On the other hand, we have, using a generalization of the Newton binomial formula,
\begin{eqnarray}\label{newton}
&& \prod_{j=1}^{P_n}\left(1-e^{-a_j}\right)
= 1-\sum_{j=1}^{P_n} e^{-a_j}
- \sum_{k=2}^{P_n} (-1)^{k+1} 
    \sum_{1\leq i_1<i_2<\cdots<i_k\leq P_n}
    e^{-\sum_{j=1}^k a_{i_j}},
\end{eqnarray}
where
$$
a_j = n\, \mathbb{P}\bigl(\|X_1-c_j\|\le u_n\bigr).
$$

Our purpose is to compare (\ref{poincare}) and (\ref{newton}).  
{\textcolor{black}{
For this purpose, we control
\[
I := \sum_{j=1}^{P_n} \mathbb{P}(A_j) 
      - \sum_{j=1}^{P_n} e^{-a_j}
\]
and
\[
II := \sum_{k=2}^{P_n} (-1)^{k+1} 
      \sum_{1 \le i_1 < \cdots < i_k \le P_n} 
      \mathbb{P}\!\left(A_{i_1} \cap \cdots \cap A_{i_k}\right)
      - \sum_{k=2}^{P_n} (-1)^{k+1} 
      \sum_{1 \le i_1 < \cdots < i_k \le P_n}
      e^{-\sum_{j=1}^{k} a_{i_j}}.
\]
}}
Since, by independence,
$$
\mathbb{P}(A_j)
= \bigl(1- \mathbb{P}(X_1\in B(c_j, u_n))\bigr)^n,
$$
and using the trivial bound for any negative real $y$, 
$$
|e^y - 1|\leq |y|,
$$
and the inequality $\ln(1-x)+x \le 0$ valid for all real $x \in (0,1)$, we obtain since
\begin{eqnarray}\label{e1} 
&& |I|\leq \sum_{j=1}^{P_n}\bigl|\mathbb{P}(A_j)- \exp\bigl({-n \mathbb{P}\left(X_1\in B(c_j, u_n)\right)}\Bigr)\Bigr| \nonumber \\
&& = \sum_{j=1}^{P_n}
    \bigl| \exp\Bigl({\,n \ln(1-\mathbb{P}(X_1\in B(c_j, u_n)))}\Bigr) 
    -  \exp\bigl({-n \mathbb{P}\left(X_1\in B(c_j, u_n)\right)}\bigr)\bigr| \nonumber \\
&& = \sum_{j=1}^{P_n} 
    e^{-n \mathbb{P}(X_1\in B(c_j, u_n))}\,
    \Bigl| 
      \exp\Bigl({\,n \Bigl(\ln(1-\mathbb{P}(X_1\in B(c_j, u_n))) 
      + \mathbb{P}(X_1\in B(c_j, u_n))\Bigr)}\Bigr) - 1
    \Bigr|, \nonumber \\
&& \leq \sum_{j=1}^{P_n} 
    e^{-n \mathbb{P}(X_1\in B(c_j, u_n))}\,
    \biggl| 
      n \ln\bigl(1-\mathbb{P}(X_1\in B(c_j, u_n))\bigr)
      + n\,\mathbb{P}(X_1\in B(c_j, u_n))
    \biggr|.
\end{eqnarray}

Recall that for any $x\in (0,1)$,
\begin{equation}\label{ln}
-x - \frac{x^2}{2(1-x)} \le \ln(1-x) \le -x.
\end{equation}
Taking $x = \mathbb{P}(X_1\in B(c_j, u_n))$ in (\ref{e1}) gives
\begin{eqnarray}\label{e2}
&& \sum_{j=1}^{P_n}
\bigl|\mathbb{P}(A_j)- e^{-n \mathbb{P}(X_1\in B(c_j, u_n))}\bigr| 
\nonumber \\
&& \le \frac{n}{2} \sum_{j=1}^{P_n}
    e^{-n \mathbb{P}(X_1\in B(c_j, u_n))}\,
    \frac{\bigl(\mathbb{P}(X_1\in B(c_j, u_n))\bigr)^2}
         {1 - \mathbb{P}(X_1\in B(c_j, u_n))}.
\end{eqnarray}

Now, under the $(a',a,b)$-assumption, it holds for all $1\le j\le P_n$,
$$
a\, u_n^b \le \mathbb{P}(X_1\in B(c_j, u_n)) 
\le a'\, u_n^b.
$$
Hence,
\begin{eqnarray*}
&& \frac{n}{2}\sum_{j=1}^{P_n} 
    e^{-n \mathbb{P}(X_1\in B(c_j, u_n))}\,
    \frac{\bigl(\mathbb{P}(X_1\in B(c_j, u_n))\bigr)^2}
         {1 - \mathbb{P}(X_1\in B(c_j, u_n))} \\
&& \le \frac{a^{\prime 2}\, n\, u_n^{2b}}{2\bigl(1 - a^{\prime}\, u_n^b\bigr)} 
     \sum_{j=1}^{P_n} e^{-n \mathbb{P}(X_1\in B(c_j, u_n))}.
\end{eqnarray*}
Therefore, by (\ref{e2}),
\begin{eqnarray}\label{e2bis}
&& \sum_{j=1}^{P_n}
\bigl|\mathbb{P}(A_j)- e^{-n \mathbb{P}(X_1\in B(c_j, u_n))}\bigr| 
\le \frac{a^{\prime 2}\, n\, u_n^{2b}}{2(1 - a^{\prime}\, u_n^b)}\,
\sum_{j=1}^{P_n} e^{-n \mathbb{P}(X_1\in B(c_j, u_n))}.
\end{eqnarray}

Our goal is to check that
$$
\limsup_{n\rightarrow \infty} 
\sum_{j=1}^{P_n} e^{-n \mathbb{P}(X_1\in B(c_j, u_n))} < \infty.
$$
Using the $(a,b)$-standard assumption and the bounds ((\ref{bound1}), (\ref{bound2})), we obtain,
\begin{eqnarray*}
&& \sum_{j=1}^{P_n} e^{-n \mathbb{P}(X_1\in B(c_j, u_n))} 
\le P_n\, e^{-n\, a\, u_n^b} \le P_n\, e^K  \frac{\ln n}{n}
\le \text{cst.}
\end{eqnarray*}
Thus, since $\lim_{n\to\infty} u_n^b=0$, we get from (\ref{e2bis}) and (\ref{bound1}),
\begin{eqnarray}\label{A3}
&& |I|\leq \sum_{j=1}^{P_n}
\bigl|\mathbb{P}(A_j)- e^{-n \mathbb{P}(X_1\in B(c_j, u_n))}\bigr|
= {O}\bigl(n\, u_n^{2b}\bigr)= {O}\bigl(\frac{\ln^2 n}{n}\bigr).
\end{eqnarray}

Our next goal is to control the quantity
\begin{eqnarray*}
&& |II|=\left|\sum_{k=2}^{P_n} (-1)^{k+1}
\sum_{1\le i_1<\cdots<i_k\le P_n}
\Bigl[
\mathbb{P}\bigl(A_{i_1}\cap \cdots \cap A_{i_k}\bigr)
-
\exp\Bigl(-n \sum_{j=1}^{k} 
\mathbb{P}\bigl(\|X_1-c_{i_j}\|\le u_n\bigr)\Bigr)
\Bigr]
\right|.
\end{eqnarray*}

By independence of the variables, and since the balls 
$\bigl(B(c_{i_j}, u_n)\bigr)_{1\le j\le P_n}$ are disjoint,
\begin{eqnarray*}
&& \mathbb{P}\bigl(A_{i_1}\cap \cdots \cap A_{i_k}\bigr) \\
&& = \bigl(\mathbb{P}\bigl(\|X_1-c_{i_1}\|>u_n,\ldots,\|X_1-c_{i_k}\|>u_n\bigr)\bigr)^n \\
&& = \Bigl(1 - \mathbb{P}\bigl(X_1\in \bigcup_{j=1}^{k} B(c_{i_j}, u_n)\bigr)\Bigr)^n \\
&& = \biggl(1 - \sum_{j=1}^{k} \mathbb{P}\bigl(X_1\in B(c_{i_j}, u_n)\bigr)\biggr)^n.
\end{eqnarray*}

Set
$$
x_k = \sum_{j=1}^{k} \mathbb{P}\bigl(\|X_1 - c_{i_j}\|\le u_n\bigr)
= \mathbb{P}\biggl(\bigcup_{j=1}^{k}
\bigl\{\|X_1-c_{i_j}\|\le u_n\bigr\}\biggr).
$$
By the $(a',a,b)$-assumption,
\begin{equation}\label{xk}
a\, k\, u_n^b \le x_k \le a'\, k\, u_n^b,\quad {\mbox{also}}\,\,\,
x_k \in [0,1].
\end{equation}

Hence
\begin{eqnarray}\label{xkbis}
&& \biggl|
\mathbb{P}\bigl(A_{i_1}\cap \cdots \cap A_{i_k}\bigr)
- \exp\Bigl(-n \sum_{j=1}^{k}
\mathbb{P}\bigl(\|X_1 - c_{i_j}\|\le u_n\bigr)\Bigr)
\biggr| {\nonumber}\\
&& = \bigl|(1 - x_k)^n - e^{-n x_k}\bigr|  {\nonumber}\\
&& \le e^{-n x_k}\,
     \bigl|\exp\bigl(n \ln(1 - x_k) + n x_k\bigr) - 1 \bigr| 
     \cdot \mathbf{1}_{\{x_k\ne 1\}}
     + e^{-n}\,\mathbf{1}_{\{x_k=1\}}.
\end{eqnarray}

We use the following bound, which is a direct consequence of (\ref{ln}):
\begin{eqnarray*}
&& e^{-n x_k}\, e^{-n\, \frac{x_k^2}{2(1 - x_k)}}
\le e^{\,n \ln(1 - x_k)} \le e^{-n x_k}.
\end{eqnarray*}
Hence, when $x_k\neq 1$,
\begin{eqnarray*}
&& e^{-n \frac{x_k^2}{2(1 - x_k)}}
\le e^{\,n \ln(1 - x_k) + n x_k} \le 1.
\end{eqnarray*}
So that
$$
0 \le 1 - e^{\,n \ln(1 - x_k) + n x_k}
\le 1 - e^{-n\, \frac{x_k^2}{2(1 - x_k)}}
\le \min\biggl(1, \, n\, \frac{x_k^2}{2(1 - x_k)}\biggr).
$$
The last bound together with (\ref{xkbis}) and (\ref{xk}), gives
\begin{eqnarray*}
&& \biggl|
(1 - x_k)^n 
- e^{-n x_k}
\biggr| \\
&& \le e^{- a\, n\, k\, u_n^b}\,
      \min\biggl(
        \frac{n\, x_k^2}{2\bigl(1 - x_k\bigr)},
        1
      \biggr)\,\mathbf{1}_{\{x_k\ne 1\}} 
     + e^{-n}\,\mathbf{1}_{\{x_k=1\}} \\
&& \le \text{cst.}\, e^{-a\, n\, k\, u_n^b}\,
      \min\biggl(
        \frac{n\, k^2\, u_n^{2b}}{2\bigl(1 - \min(1,\, k\, u_n^b\, a')\bigr)},
        1
      \biggr)
      + e^{-n}.
\end{eqnarray*}

Consequently, from the last bound  together with (\ref{xkbis}) and (\ref{bound1}), we obtain, 
\begin{eqnarray*}
&& \left|
\sum_{k=2}^{P_n} (-1)^{k+1}
\sum_{1\le i_1<\cdots<i_k\le P_n}
\biggl[
\mathbb{P}\bigl(A_{i_1}\cap \cdots \cap A_{i_k}\bigr)
- \exp\Bigl(-n \sum_{j=1}^{k}
\mathbb{P}\bigl(\|X_1 - c_{i_j}\|\le u_n\bigr)\Bigr)
\biggr]
\right| \\
&& 
\sum_{k=2}^{P_n} \binom{P_n}{k}\,
\biggl|
\mathbb{P}\bigl(A_{i_1}\cap \cdots \cap A_{i_k}\bigr)
- \exp\Bigl(-n \sum_{j=1}^{k}
\mathbb{P}\bigl(\|X_1 - c_{i_j}\|\le u_n\bigr)\Bigr)
\biggr|\\
&& \le \text{cst.}\sum_{k=2}^{P_n}
     \binom{P_n}{k}\,
     e^{-a\, n\, k\, u_n^b}\,
     \min\biggl(
       \frac{n\, k^2\, u_n^{2b}}{2\bigl(1 - \min(1,\, k\, u_n^b\, a')\bigr)},
       1
     \biggr) 
     + e^{-n}\, \sum_{k=2}^{P_n} \binom{P_n}{k} \\
&& \le \text{cst.}\sum_{k=2}^{P_n}
     \binom{P_n}{k}\,\left(\frac{\ln n}{n}\right)^k
    \,
     \min\biggl(
       \frac{n\, k^2\, u_n^{2b}}{2\bigl(1 - \min(1,\, k\, u_n^b\, a')\bigr)},
       1
     \biggr) 
     + e^{-n}\, {2}^{P_n}.     
\end{eqnarray*}

Using a Stirling-type bound bound on binomial coefficients together with (\ref{bound2}),
$$
\binom{P_n}{k} \le \left(\frac{e\, P_n}{k}\right)^k \leq \left(\frac{c\, n}{k\ln (n)}\right)^k,
$$
for some positive constant $c$, we obtain, 
\begin{eqnarray}\label{BBound}
&& \left|
\sum_{k=2}^{P_n} (-1)^{k+1}
\sum_{1\le i_1<\cdots<i_k\le P_n}
\biggl[
\mathbb{P}\bigl(A_{i_1}\cap \cdots \cap A_{i_k}\bigr)
- \exp\Bigl(-n \sum_{j=1}^{k}
\mathbb{P}\bigl(\|X_1 - c_{i_j}\|\le u_n\bigr)\Bigr)
\biggr]
\right| {\nonumber}\\
&& \le \text{cst.}\sum_{k=2}^{P_n}
     \, \left(\frac{c}{k}\right)^k
    \,
     \min\biggl(
       \frac{n\, k^2\, u_n^{2b}}{2\bigl(1 - \min(1,\, k\, u_n^b\, a')\bigr)},
       1
     \biggr) 
     + {\mbox{cst.}}\,\, e^{-n/2}.     
\end{eqnarray}
Now (\ref{bound1}) gives,  for $n$ large enough,
\begin{eqnarray*}
&& \sum_{k=2}^{P_n}
\left(\frac{c}{k}\right)^k
\min\biggl(
  \frac{n\, k^2\, u_n^{2b}}{2\bigl(1 - \min(1,\, k\, u_n^b\, a')\bigr)},
  1
\biggr) \\
&& \le \frac{\ln^2 n}{n}
    \sum_{k=1}^{\lfloor \sqrt{n}/\ln n \rfloor}
    \left(
      \frac{c}{k}
    \right)^k
    k^2 \quad + \sum_{k=\lfloor \sqrt{n}/\ln n \rfloor}^{P_n}
    \left(
      \frac{c}{k}
    \right)^k.
\end{eqnarray*}
We have,
\begin{eqnarray*}
    && \sum_{k=1}^{\lfloor \sqrt{n}/\ln n \rfloor}
\left(
   \frac{c}{k}
\right)^k
k^2
\le \text{cst.} + \sum_{k=[c]+1}^{\lfloor \sqrt{n}/\ln n \rfloor}
\left(
   \frac{c}{[c]+1}
\right)^k
k^2 \le \text{cst.}
\end{eqnarray*}
Therefore,
$$
\frac{\ln^2 n}{n}\sum_{k=1}^{\lfloor \sqrt{n}/\ln n \rfloor}
\left(\frac{c}{k}\right)^k
k^2 = {\mathcal O}\left(\frac{\ln^2 n}{n}\right).
$$
Moreover,  we have
\begin{eqnarray*}
&& \sum_{k=\lfloor \sqrt{n}/\ln n \rfloor}^{P_n}\left(
  \frac{c}{k}
\right)^k \le
\sum_{k=\lfloor \sqrt{n}/\ln n \rfloor}^{P_n}\exp\bigl(-k \ln(k/c)\bigr) ={\mathcal O}\left(\exp(-\frac{\sqrt{n}}{2})\right).
\end{eqnarray*}
Hence, combining all previous bounds, we finally deduce
\begin{equation*}
\sum_{k=2}^{P_n}
\left(\frac{c}{k}\right)^k
\min\left(
  \frac{n\, k^2\, u_n^{2b}}{2(1 - \min(1,\, k\, u_n^b\, a'))},
  1
\right)
\le \text{cst.} \frac{\ln^2 n}{n}.
\end{equation*}
We finally conclude thanks to (\ref{BBound})
\begin{eqnarray*}
&& \left|\sum_{k=2}^{P_n}
(-1)^{k+1}
\sum_{1\le i_1<\cdots<i_k\le P_n}
\mathbb{P}\bigl(A_{i_1}\cap \cdots \cap A_{i_k}\bigr)
-
\exp\left(-n \sum_{j=1}^{k}
\mathbb{P}\bigl(\|X_1 - c_{i_j}\|\le u_n\bigr)\right)
\right| \\
&&
= {{O}}\left(\frac{\ln^2 n}{n}\right)
\end{eqnarray*}

Together with (\ref{A3}) ((\ref{poincare}) and (\ref{newton})), this implies the desired approximation
\begin{equation*}
\mathbb{P}\left(
\max_{1\le j\le P_n}
\min_{1\le i\le n}
\|X_i - c_j\| \le u_n
\right)
=
\prod_{j=1}^{P_n}
\left(1 - e^{-a_j}\right) 
+ {{ O}}\left(\frac{\ln^2 n}{n}\right),
\end{equation*}
where $a_j = n\, \mathbb{P}(\|X_1 - c_j\|\le u_n)$.
This completes the proof of Lemma \ref{propak}.

\subsection{Proof of Proposition \ref{main}}
Proposition~\ref{main} is an immediate consequence of Lemma \ref{propak} together with the following lemma.

\begin{lem}\label{propak1}
Let $X_1$ be a compactly supported random variable satisfying 
the requirement of Definition \ref{order}. Let $\mathbb M$ denote its compact support. Let $(u_n)_n$ be a sequence of positive real numbers satisfying Condition~\eqref{conditionun}, and let $(c_i)_{1\le i \le P_n}$ be a $u_n$-packing of $\mathbb M$.
Then,
\[
\lim_{n \to \infty}\left|  
\prod_{j=1}^{P_n} \Bigl( 1 - \exp\bigl( -n \, \mathbb{P}\left(\|X_1 - c_j\| \le u_n\right) \bigr) \Bigr)
- \exp\left(- P_n \exp\bigl(-n u_n^{\,b}\, a\bigr)\right)
\right| = 0.
\]
\end{lem}

\subsubsection{Proof of Lemma \ref{propak1}}

We start by using the following two bounds:
\begin{equation}\label{bound3}
\left| \prod_{j=1}^m u_j - \prod_{j=1}^m v_j \right|
\;\leq\; \sum_{j=1}^m |u_j - v_j|,
\end{equation}
which holds for any positive integer $m$ and for any $0 \leq u_j, v_j \leq 1$, and
\[
|e^{-x} - e^{-y}| \;\le\; e^{-x} \, |x - y|,
\]
for positive $x$ and $y$ satisfying $y \ge x$.
Recall also that, by the $(a,b)$-standard assumption,
\[
\mathbb{P}\bigl(X_1 \in B(c_j, u_n)\bigr) \;\ge\; a\, u_n^b.
\]

Therefore,
\begin{eqnarray*}
&& \left|
\prod_{j=1}^{P_n} \bigl(1 - e^{-n\, \mathbb{P}(X_1 \in B(c_j, u_n))}\bigr)
-
\prod_{j=1}^{P_n} \bigl(1 - \exp(-n\, u_n^b\, a)\bigr)
\right| \\
&& \le
\sum_{j=1}^{P_n}
\left|
\exp\bigl(-n\, u_n^b\, a\bigr)
- e^{-n\, \mathbb{P}(X_1 \in B(c_j, u_n)) }
\right| \\
&& \le
n \sum_{j=1}^{P_n}
\exp\bigl(-n\, u_n^b\, a\bigr)
\bigl|
u_n^b\, a
- \mathbb{P}(X_1 \in B(c_j, u_n))
\bigr|.
\end{eqnarray*}

From Definition \ref{order}, we have for all $1 \le j \le P_n$:
\[
\bigl|
\mathbb{P}(X_1 \in B(c_j, u_n)) - a\, u_n^b
\bigr|
\;\le\;
d\, u_n^{b+\alpha}.
\]

Therefore, using bounds (\ref{bound1}) and (\ref{bound2}), we obtain
\begin{eqnarray*}
&& \left|
\prod_{j=1}^{P_n}
\bigl(1 - e^{-n\, \mathbb{P}(X_1 \in B(c_j, u_n))}\bigr)
-
\prod_{j=1}^{P_n}
\bigl(1 - \exp(-n\, u_n^b\, a)\bigr)
\right| \\
&& \le
d\, n\, u_n^{b+\alpha}\, P_n \, \exp\bigl(-n\, u_n^b\, a\bigr) \\
&& \le
\text{cst.} \; n\, u_n^{\alpha}\, \frac{\ln n}{n}
\;=\;
\text{cst.} \; (\ln n)\, u_n^{\alpha} \\
&& \le
\text{cst.} \; (\ln n)\,\left(\frac{\ln n}{n}\right)^{\alpha/b}.
\end{eqnarray*}

We deduce that, since $\alpha > 0$,
\begin{equation}\label{lim1}
\lim_{n \to \infty}
\left|
\prod_{j=1}^{P_n}
\bigl(1 - e^{-n\, \mathbb{P}(X_1 \in B(c_j, u_n))}\bigr)
-
\bigl(1 - \exp(-n\, u_n^b\, a)\bigr)^{P_n}
\right|
=
0.
\end{equation}

Next, applying (\ref{bound3}), we have
\begin{eqnarray*}
&& \left|
\bigl(1 - \exp(-n\, u_n^b\, a)\bigr)^{P_n}
-
\exp\bigl(-P_n\, \exp(-n\, u_n^b\, a)\bigr)
\right| \\
&& \le
P_n \;
\left|
1 - \exp(-n\, u_n^b\, a)
- \exp\bigl(-\exp(-n\, u_n^b\, a)\bigr)
\right|.
\end{eqnarray*}

Recall that for any $x \in [0,1)$,
\[
\bigl|\, 1 - x - e^{-x} \bigr|
\;\le\;
\frac{x^2}{2\, (1 - x)}.
\]

Hence, taking $x = \exp(-n\, u_n^b\, a)$ in the inequality above, we obtain
\[
\left|
1 - \exp(-n\, u_n^b\, a)
- \exp\bigl(-\exp(-n\, u_n^b\, a)\bigr)
\right|
\;\le\;
\frac{\exp(-2\, n\, u_n^b\, a)}{2 \bigl( 1 - \exp(-n\, u_n^b\, a)\bigr )}.
\]

Therefore, by (\ref{bound1}) and (\ref{bound2}),
\begin{eqnarray*}
&& \left|
\bigl(1 - \exp(-n\, u_n^b\, a)\bigr)^{P_n}
-
\exp\bigl(-P_n\, \exp(-n\, u_n^b\, a)\bigr)
\right| \\
&& \le
P_n \; \frac{\exp(-2\, n\, u_n^b\, a)}{2 \bigl( 1 - \exp(-n\, u_n^b\, a)\bigr )} \\
&& \le
\text{cst.} \;
\frac{n}{\ln n} \times
\frac{\bigl( \ln n / n \bigr)^2}{1 - \text{cst.} \; \ln n / n} \\
&& \le
\text{cst.} \;
\frac{\ln n}{n}.
\end{eqnarray*}

Consequently,
\begin{equation}\label{lim2}
\lim_{n \to \infty}
\left|
\bigl(1 - \exp(-n\, u_n^b\, a)\bigr)^{P_n}
-
\exp\bigl( - P_n\, \exp(-n\, u_n^b\, a) \bigr)
\right|
=
0.
\end{equation}

The proof of Lemma \ref{propak1} is thus complete, by combining (\ref{lim1}) and (\ref{lim2}).

\subsection{Proof of Corollary \ref{main2}}
Our first goal is to prove that the sequence \((U_n(u))\), defined in (\ref{lambert}), satisfies condition (\ref{conditionun}).
It is well known (\cite{Corless}) that the Lambert function $W$ satisfies,
\[
W(y) - \ln(y) + \ln(\ln(y)) 
= \frac{\ln(\ln(y))}{\ln(y)} + o\left(\frac{\ln(\ln(y))}{\ln(y)}\right),
\]
as \(y\) tends to infinity. 
Consequently, using (\ref{lambert}), we obtain
\[
\lim_{n \rightarrow \infty} \left| n\, a\, U_n^b(u) - \left[ \ln(n) - \ln(\ln(n)) + u + \ln\bigl(a\, c(\mathbb M)\bigr) \right] \right| = 0.
\]
This implies that the sequence 
\[
\left| n\, a\, U_n^b(u) - \bigl(\ln(n) - \ln(\ln(n))\bigr) \right|
\]
is bounded, and condition (\ref{conditionun}) is therefore satisfied. 
Hence, we may apply Proposition \ref{main}, yielding
\begin{equation}\label{lambert2}
\lim_{n \rightarrow \infty}
\left|
\mathbb{P}\left( \max_{1 \le j \le P_n} \min_{1 \le i \le n} \| X_i - c_j \| \le U_n(u) \right)
- 
\exp\left(- P_n \exp\bigl(-n\, U_n^b(u)\, a \bigr)\right)
\right|
= 0.
\end{equation}
Our next goal is to prove that, under Assumption ~\hyperref[assumpP]{${(\mathcal P)}$},
\begin{equation}\label{lambert1}
\lim_{n \rightarrow \infty}
\exp\left(- P_n \exp\bigl(-n\, U_n^b(u)\, a\bigr)\right)
=
\exp\bigl(-\exp(-u)\bigr).
\end{equation}
Using Assumption~\hyperref[assumpP]{${(\mathcal P)}$}, we have
\begin{eqnarray}\label{bound4}
&& P_n \exp\bigl(-n\, U_n^b(u)\, a\bigr) = \exp(\ln(P_n))\exp\bigl(-n\, U_n^b(u)\, a\bigr) \nonumber \\
&& = \exp\bigl(\frac{\ln(P_n)}{\ln(c(\mathbb M)\, U_n^{-b}(u))+ o(1)}\bigr)\exp\bigl({\ln(c(\mathbb M)\, U_n^{-b}(u))+ o(1)}\bigr)\exp\bigl(-n\, U_n^b(u)\, a\bigr) \nonumber \\
&& = \exp\bigl(\frac{\ln(P_n)}{\ln(c(\mathbb M)\, U_n^{-b}(u))+ o(1)}\bigr) c(\mathbb M)\, U_n^{-b}(u)) \exp\bigl(-n\, U_n^b(u)\, a + o(1)\bigr)
\end{eqnarray}
Moreover, recall that \(W(y)\, e^{W(y)} = y\) for any positive \(y\). Thus,
\begin{eqnarray*}
&&
c(\mathbb M)\, U_n^{-b}(u) \exp\bigl(-n\, U_n^b(u)\, a\bigr) \\
&&=
c(\mathbb M)\, U_n^{-b}(u)
\exp\left(
- n\, U_n^b(u)\, a
+ W\bigl( n\, a\, c(\mathbb M)\, e^u \bigr)
\right)
\exp\left(- W\bigl( n\, a\, c(\mathbb M)\, e^u \bigr)\right) \\
&&=
c(\mathbb M)\, U_n^{-b}(u)\, \exp\left(
- n\, U_n^b(u)\, a
+ W\bigl( n\, a\, c(\mathbb M)\, e^u \bigr)
\right)
\frac{W\bigl( n\, a\, c(\mathbb M)\, e^u \bigr)}{n\, a\, c(\mathbb M)\, e^u}
 \\
&&\sim 
\frac{
W\bigl( n\, a\, c(\mathbb M)\, e^u \bigr)
}{
n\, a\, U_n^b(u)
}\,
e^{-u},
\end{eqnarray*}
as $n$ tends to infinity by (\ref{lambert}).
It is known that \(\lim_{n \rightarrow \infty} W\bigl( n\, a\, c(\mathbb M)\, e^u \bigr) = +\infty\). Consequently, from (\ref{lambert}) we obtain
\[
\lim_{n \rightarrow \infty}
\frac{
n\, a\, U_n^b(u)
}{
W\bigl( n\, a\, c(\mathbb M)\, e^u \bigr)
}
= 1,
\]
and hence
\begin{equation}\label{bound5}
\lim_{n \rightarrow \infty}
c(\mathbb M)\, U_n^{-b}(u) \exp\bigl(-n\, U_n^b(u)\, a\bigr)
=
e^{-u}.
\end{equation}
Combining this with (\ref{bound4}) and Assumption~\hyperref[assumpP]{${(\mathcal P)}$}, we deduce that
$$
\lim_{n \rightarrow \infty}P_n \exp\bigl(-n\, U_n^b(u)\, a\bigr)= e^{-u}
$$
and therefore the limit (\ref{lambert1}) is proved.
The proof of Corollary \ref{main2} is thus complete, combining (\ref{lambert2}) and (\ref{lambert1}).

\subsection{Proof of Theorem \ref{promain}}
Let $u \in \mathbb{R}$ be fixed. Define
\[
U_n(u) = \alpha_n u + \beta_n,
\]
where the sequences $(\alpha_n)_n$ and $(\beta_n)_n$ are as defined in the statement of Theorem \ref{promain}. 
We have
\begin{align*}
n^{1/b} \, a^{1/b} \, U_n(u) 
&= (\ln(n))^{1/b} \left(1 + \frac{1}{b \, \ln(n)} \bigl( u + \ln(a \, c(\mathbb M)) - \ln(\ln(n)) \bigr)\right).
\end{align*}

Since 
\[
(1+x)^b = 1 + b\,x + O(x^2), 
\]
as $x$ tends to $0$, we obtain
\begin{align*}
n\, a \, U_n^b(u) 
&= (\ln(n)) \left( 1 + \frac{1}{b \, \ln(n)} \bigl( u + \ln(a \, c(\mathbb M)) - \ln(\ln(n)) \bigr) \right)^b \\
&= \ln(n) - \ln(\ln(n)) + u + \ln(a \, c(\mathbb M)) + O\left(\frac{1}{\ln n}\right).
\end{align*}

Finally,
\[
\lim_{n \to \infty} \left| n a \, U_n^b(u) - W\bigl( n\, a\, c(\mathbb M)\, e^{u} \bigr) \right| = 0.
\]
Condition (\ref{lambert}) is then satisfied. 
Let $c_1, \dots, c_{P_n}$ be a maximal $U_n(u)$-packing of $\mathbb M$. Then, by Corollary \ref{main2}, we have
\[
\lim_{n \to \infty} \mathbb{P}\left( \max_{1 \le j \le P_n} \min_{1 \le i \le n} \| X_i - c_j \| \le U_n(u) \right)
= \exp\bigl( - \exp(-u) \bigr),
\]
or equivalently
$$
\lim_{n \rightarrow \infty} 
\mathbb{P}\left(
\bigl\{c_1,\cdots, c_{P_n}\bigr\} \subset \bigcup_{i=1}^n B(X_i, \alpha_n\,u+ \beta_n)\right)
=
\exp\bigl(- \exp(-u)\bigr).
$$
The proof of Theorem \ref{promain} is complete.
\subsection{Proof of Corollary \ref{maintheo}}
We should apply Theorem \ref{promain}.
Let $u \in \mathbb{R}$ be fixed. Define
$
U_n(u) = \alpha_n u + \beta_n,
$
where the sequences $(\alpha_n)_n$ and $(\beta_n)_n$ are as defined in the statement of Theorem \ref{promain}.
Clearly, the event 
$$
\bigl\{c_1,\cdots, c_{P_n}\bigr\} \subset \bigcup_{i=1}^n B(X_i, \alpha_n\,u+ \beta_n)
$$
implies that,
$$
\bigcup_{j=1}^{P_n} B(c_j, \alpha_n\,u+ \beta_n) \subset \bigcup_{i=1}^n B(X_i, 2(\alpha_n\,u+ \beta_n)).
$$
So that,
\begin{eqnarray*}    
&&\mathbb P\left(\bigl\{c_1,\cdots, c_{P_n}\bigr\} \subset \bigcup_{i=1}^n B(X_i, \alpha_n\,u+ \beta_n)\right)  \\
&&\leq\mathbb P\left( \bigcup_{j=1}^{P_n} B(c_j, \alpha_n\,u+ \beta_n) \subset \bigcup_{i=1}^n B(X_i, 2(\alpha_n\,u+ \beta_n)) \right).
\end{eqnarray*}
Consequently by Theorem \ref{promain}, we obtain,
$$
\exp\bigl( - \exp(-u) \bigr)
\le
\liminf_{n \to \infty} \mathbb P\left( \bigcup_{j=1}^{P_n} B(c_j, \alpha_n\,u+ \beta_n) \subset \bigcup_{i=1}^n B(X_i, 2(\alpha_n\,u+ \beta_n)) \right).
$$
Now, since,
\[
\max_{1 \le j \le P_n} \min_{1 \le i \le n} \| X_i - c_j \| 
\le d_H(\bbx_n, \mathbb M) \quad \text{a.s.},
\]
we have
\[
\mathbb{P}\left( d_H(\bbx_n, \mathbb M) \le U_n(u) \right)
\le 
\mathbb{P}\left( \max_{1 \le j \le P_n} \min_{1 \le i \le n} \| X_i - c_j \| \le U_n(u) \right).
\]
Hence,
\[
\limsup_{n \to \infty} 
\mathbb{P}\left( d_H(\bbx_n, \mathbb M) \le U_n(u) \right)
\le \exp\bigl( - \exp(-u) \bigr).
\]

Our goal now is to prove that
\[
\exp\bigl( - \exp(-u) \bigr)
\le
\liminf_{n \to \infty} 
\mathbb{P}\left( d_H(\bbx_n, \mathbb M) \le 3 \, U_n(u) \right).
\]

For this, we first recall the following lemma, which is well-known in the literature (see, for instance, \cite{Eddie}).

\begin{lem}\label{lemtool}
Let $c_1, \dots, c_{P}$ be an $\varepsilon$-maximal packing of $\mathbb M$. Then
\[
\mathbb M \subset \bigcup_{j=1}^P B(c_j, 2\varepsilon).
\]
If $\bbx_n = \{ X_1, \dots, X_n \} \subset \mathbb M$, then
\[
d_H(\bbx_n, \mathbb M) 
\le 2 \varepsilon 
+ \max_{1 \le j \le P} \min_{1 \le i \le n} \| X_i - c_j \|.
\]
\end{lem}

Let $c_1, \dots, c_{P_n}$ be a $U_n(u)$-maximal packing of $\mathbb M$. Then, applying Lemma \ref{lemtool} with $\varepsilon = U_n(u)$, we obtain almost surely:
\begin{equation}\label{H1}
d_H(\bbx_n, \mathbb M) 
\le 2\, U_n(u) 
+ \max_{1 \le j \le P_n} \min_{1 \le i \le n} \| X_i - c_j \|.
\end{equation}

Consequently,
\[
\mathbb{P}\left( \max_{1 \le j \le P_n} \min_{1 \le i \le n} \| X_i - c_j \| \le U_n(u) \right)
\le
\mathbb{P}\left( d_H(\bbx_n, \mathbb M) \le 3\, U_n(u) \right),
\]
and,
$$
\exp(-\exp(-u))\leq \liminf_{n\rightarrow \infty}\mathbb{P}\left( d_H(\bbx_n, \mathbb M) \le 3\, U_n(u) \right).
$$
The proof is complete thanks to Corollary \ref{main2}.

\subsection{Proof of Proposition \ref{theopacking}}\label{proofPro}

Let $c_1,\dots,c_P$ be an $\varepsilon$-packing of $\mathbb{M}$ (not necessarily maximal). Arguing as in (\ref{P1}), we obtain 
\[
P \min_{1\le j\le P} \mathbb{P}\bigl(X \in B(c_j,\varepsilon)\bigr)
\le 
\sum_{j=1}^{P} \mathbb{P}\bigl(X \in B(c_j,\varepsilon)\bigr)
=
\mathbb{P}\Bigl(X \in \bigcup_{j=1}^{P} B(c_j,\varepsilon)\Bigr)
\le 1.
\]
Therefore,
\[
P \le 
\frac{1}{\min_{1\le j\le P} 
\mathbb{P}\bigl(X \in B(c_j,\varepsilon)\bigr)}.
\]
In particular,
\[
P(\mathbb{M},\varepsilon)
\le 
\frac{1}{\min_{1\le j\le P(\mathbb{M},\varepsilon)} 
\mathbb{P}\bigl(X \in B(c_j,\varepsilon)\bigr)}.
\]

Now let $P=P(\mathbb{M},\varepsilon)$ and let $c_1,\dots,c_P$ be a maximal $\varepsilon$-packing. 
Then, arguing as in (\ref{P2}), we obtain 
\[
1
\le 
P(\mathbb{M},\varepsilon)
\max_{1\le j\le P}
\mathbb{P}\bigl(X \in B(c_j,2\varepsilon)\bigr).
\]
Thus,
\[
\frac{1}{\max_{1\le j\le P(\mathbb{M},\varepsilon)}
\mathbb{P}\bigl(X \in B(c_j,2\varepsilon)\bigr)}
\le 
P(\mathbb{M},\varepsilon)
\le
\frac{1}{\min_{1\le j\le P(\mathbb{M},\varepsilon)}
\mathbb{P}\bigl(X \in B(c_j,\varepsilon)\bigr)}.
\]
\\
Taking logarithms and dividing by $-\ln\varepsilon$, we obtain, letting
\[
\overline{\Lambda}(\varepsilon)
=
\frac{
\ln\!\left(
\min_{1\le j\le P(\mathbb{M},\varepsilon)}
\mathbb{P}(X \in B(c_j,\varepsilon))
\right)
}{\ln\varepsilon},
\qquad
\underline{\Lambda}(\varepsilon)
=
\frac{
\ln\!\left(
\max_{1\le j\le P(\mathbb{M},\varepsilon)}
\mathbb{P}(X \in B(c_j,\varepsilon))
\right)
}{\ln\varepsilon}.
\]
\\
\begin{equation}\label{E1}
\frac{\ln(2\varepsilon)}{\ln(\varepsilon)}\,
\underline{\Lambda}(2\varepsilon)
\le
\frac{\ln P(\mathbb{M},\varepsilon)}{-\ln\varepsilon}
\le \overline{\Lambda}(\varepsilon).
\end{equation}
\\

We shall use the following technical lemma.

\begin{lem}\label{proPackingNumber}
Assume that
\begin{equation}\label{E2}
\overline{\Lambda}(\varepsilon)
=
A + \frac{B}{|\ln\varepsilon|}
+ o\!\left(\frac{1}{|\ln\varepsilon|}\right),
\qquad
\underline{\Lambda}(\varepsilon)
=
A + \frac{B}{|\ln\varepsilon|}
+ o\!\left(\frac{1}{|\ln\varepsilon|}\right).
\end{equation}
Then
\[
\frac{\ln P(\mathbb{M},\varepsilon)}
{\ln\bigl(\varepsilon^{-A} e^B\bigr) + o(1)}
\longrightarrow 1
\quad \text{as } \varepsilon \to 0.
\]
\end{lem}

\paragraph{Proof of Lemma \ref{proPackingNumber}.}

Combining \eqref{E1} with the assumed expansions (\ref{E2}) yields
\[
\frac{\ln(2\varepsilon)}{\ln(\varepsilon)}
\left(
A + \frac{B}{|\ln(2\varepsilon)|}
+ o\!\left(\frac{1}{|\ln(2\varepsilon)|}\right)
\right)
\le
-
\frac{\ln P(\mathbb{M},\varepsilon)}{\ln\varepsilon}
\le
A + \frac{B}{|\ln\varepsilon|}
+ o\!\left(\frac{1}{|\ln\varepsilon|}\right).
\]

Since, as $\varepsilon \to 0$
\[
\frac{\ln(2\varepsilon)}{\ln\varepsilon} \to 1
\quad\text{and}\quad
\frac{
A + \frac{B}{|\ln(2\varepsilon)|}
+ o\left(\frac{1}{|\ln (2\varepsilon)|}\right)
}{
A + \frac{B}{|\ln\varepsilon|}
+ o\left(\frac{1}{|\ln\varepsilon|}\right)
}
\to 1,
\]
we obtain, as $\varepsilon \to 0$
\[
-\frac{\ln P(\mathbb{M},\varepsilon)}{\ln\varepsilon}
\left(
A + \frac{B}{|\ln\varepsilon|}
+ o\!\left(\frac{1}{|\ln\varepsilon|}\right)\right)^{-1} \to 1.
\]
The last limit together with the equality,
\[
\ln\!\left(\frac{1}{\varepsilon}\right)
\left(
A + \frac{B}{|\ln\varepsilon|}
+ o\!\left(\frac{1}{|\ln\varepsilon|}\right)
\right)
=
\ln\!\left(\varepsilon^{-A} e^B\right)
+ o(1),
\]
gives, as $\varepsilon \to 0$, 
\[
\frac{\ln P(\mathbb{M},\varepsilon)}{
\ln\!\left(\varepsilon^{-A} e^B\right)
+ o(1)} \to 1,
\]
which proves the lemma.
\hfill $\square$

\medskip

We now return to the proof of Proposition \ref{theopacking}. 
It remains to verify the expansions above.

Let $X$ be uniformly distributed on $\mathbb{M}$, then
\[
\mathbb{P}\bigl(X \in \mathbb{M}\cap B(x,\varepsilon)\bigr)
=
\frac{
\mathrm{Vol}\bigl(\mathbb{M}\cap B(x,\varepsilon)\bigr)
}{
\mathrm{Vol}(\mathbb{M})
}.
\]

By Theorem \ref{theoK},
\[
\mathbb{P}\bigl(X \in \mathbb{M}\cap B(x,\varepsilon)\bigr)
=
\frac{\sigma_{p-1}}{p\,\mathrm{Vol}(\mathbb{M})}
\varepsilon^p
\left(
1 + C_x \varepsilon^2 + O(\varepsilon^3)
\right),
\]
where
\[
C_x
=
\frac{1}{8(p+2)}
\left(
2\|B_x\|^2 - \|H_x\|^2
\right).
\]

Taking logarithms,
\[
\ln \mathbb{P}\bigl(X \in \mathbb{M}\cap B(x,\varepsilon)\bigr)
=
p\ln\varepsilon
+
\ln\!\left(
\frac{\sigma_{p-1}}{p\,\mathrm{Vol}(\mathbb{M})}
\right)
+
O(\varepsilon^2).
\]

The remainder is uniform in $x$ because
$|2\|B_x\|^2 - \|H_x\|^2|$ is uniformly bounded by assumption. Hence,
for every $x\in\mathbb{M}$,
\[
\frac{
\ln \mathbb{P}\bigl(X \in \mathbb{M}\cap B(x,\varepsilon)\bigr)
}{
\ln\varepsilon
}
=
p
-
\frac{
\ln\!\left(
\frac{\sigma_{p-1}}{p\,\mathrm{Vol}(\mathbb{M})}
\right)
}{|\ln\varepsilon|}
+
o\!\left(\frac{1}{|\ln\varepsilon|}\right).
\]

Applying this to the points realizing the minimum and maximum in the definition of $\overline{\Lambda}(\varepsilon)$ and $\underline{\Lambda}(\varepsilon)$, Lemma \ref{proPackingNumber} applies with
\[
A=p,
\qquad
B=
\ln\!\left(
\frac{p\,\mathrm{Vol}(\mathbb{M})}{\sigma_{p-1}}
\right).
\]

The conclusion of the lemma completes the proof of Proposition \ref{theopacking}.
\hfill $\square$
\\
\\
{\bf{Acknowledgements.}}
We are grateful to Eddie Aamari for bringing reference~\cite{Eddie} to our attention. 
We also thank Boris Thibert for stimulating discussions regarding the computation of volumes of extrinsic Euclidean balls. 
{\textcolor{black}{We further thank Mathew D. Penrose for drawing our attention to Proposition 2.11 in \cite{Penrose2025b}, which is closely related to our Corollary 3.5. }}
Finally, we would like to express our sincere thanks to the Editor, the Associate Editor, and the two reviewers for their insightful and constructive comments, which have substantially improved the paper.

\section*{Author Contributions}
The author  contributed to conducting the research and writing the manuscript.

\section*{Funding}
This research received no specific grant from any funding agency in the public, commercial, or not-for-profit sectors.

\section*{Data Availability}
No datasets were generated or analysed during the current study.

\section*{Declarations}

\subsection*{Competing Interests}
The author declares no competing interests.
\subsection*{Ethics Approval}
Not applicable.

\end{document}